\documentclass[11pt]{amsart}
\usepackage{float}
\usepackage[dvipsnames]{xcolor}
\usepackage{soul}
\usepackage{pgfplots}
\usepackage{dsfont}
\usepackage{bbm}

\definecolor{Cream}{RGB}{255,253,208}
\definecolor{Deepblue}{RGB}{0,153,153}
\definecolor{Deepred}{RGB}{204,0,102}
\definecolor{Creamblue}{RGB}{0,255,255}
\definecolor{Creamred}{RGB}{255,51,153}
\definecolor{Deeppurple}{RGB}{110,0,110}
\definecolor{Creampurple}{RGB}{204,0,204}
\definecolor{Redpurple}{RGB}{210,25,150}
\definecolor{Bluepurple}{RGB}{150,25,240}

\usepackage[utf8]{inputenc}
\usepackage{amsmath}
\usepackage{mathtools}
\usepackage{amsthm}
\usepackage{amsfonts}
\usepackage{amssymb}
\usetikzlibrary{calc}

\usepackage[margin=1in]{geometry}
\usepackage{tikz}
\usepackage{tikz-cd}
\usetikzlibrary{positioning}
  
\usepackage{hyperref}
\usepackage[capitalize]{cleveref}
\usepackage[shortlabels,inline]{enumitem}

\makeatletter
\def\l@subsection{\@tocline{2}{0.2em}{2.3em}{}{}}
\makeatother

\numberwithin{equation}{section}

\DeclarePairedDelimiterX{\norm}[1]{\lVert}{\rVert}{#1}

\newtheorem{theorem}{Theorem}[section]
\newtheorem*{theorem*}{Theorem}
\newtheorem{lemma}[theorem]{Lemma}
\newtheorem{proposition}[theorem]{Proposition}
\newtheorem*{proposition*}{Proposition}

\newtheorem{corollary}[theorem]{Corollary}
\theoremstyle{definition}
\newtheorem{example}[theorem]{Example}

\newtheorem{definition}[theorem]{Definition}
\newtheorem{remark}[theorem]{Remark}

\newcommand{\B}{\mathbb{B}}

\newcommand{\M}{\mathcal{M}}
\newcommand{\K}{\mathcal{K}}

\usepackage{biblatex}

\begin{document}

\begin{abstract}

Certain classes of multiparameter persistence modules may be encoded as signed barcodes, represented as points in a polyhedral subset of Euclidean space, we refer to as signed persistence diagrams. These signed persistence diagrams exist in the dual space of compactly supported, Lipschitz functionals on a polyhedral pair.  In the interest of statistics and machine learning on multiparameter persistence modules, we aim to embed these signed persistence diagrams into Banach or Hilbert space. We use iteratively refined triangulations to define a Schauder Basis of compactly supported Lipschitz functionals. Evaluation of these functionals embeds signed persistence diagrams into the space of real-valued sequences.  Furthermore, we show that in the larger space of relative Radon measures, the Schauder basis we have defined is minimal to induce an embedding.
\end{abstract}

\title{ A Schauder Basis for Multiparameter Persistence}
\author{Peter Bubenik}
\address{Department of Mathematics, University of Florida, Gainesville, USA}
\email{peter.bubenik@ufl.edu}
\author{Zachariah Ross}
\address{Department of Mathematics, University of Florida, Gainesville, USA}
\email{thomas.z@ufl.edu}


\maketitle

\setcounter{tocdepth}{2}
\tableofcontents

\section{Introduction}


Topological Data Analysis (TDA) is a broad field of study that uses various techniques to extract and encode the shape of data. Working under the assumption that this data is sampled from some unknown probability distribution on Euclidean space, we may apply these techniques to understand topological features of the distribution.

For many tools used in TDA, we begin by applying a filtration function, such as Vietoris Rips or a height function, to a data set in Euclidean space. This maps the data to a filtration of topological spaces, i.e. a functor from a poset $P$ into $TOP$. Choosing a degree $k$ and composing this functor with the homology functor $H_k(-, \mathbb{F})$ yields a functor from poset $P$ into the category $vec_\mathbb{F}$ of finite dimensional vector spaces over field $\mathbb{F}$. 

Functors of this type are called \emph{persistence modules}, or more specifically, \emph{1-parameter persistence modules} when $P$ is totally ordered. The \emph{barcode} and \emph{persistence diagram} are used as descriptors of 1-parameter persistence modules. These provide multiscale topological information from the sampling. 

When 1-parameter persistence modules are derived from a filtration of data, the resulting persistence modules, and by extension the persistence diagrams, can be quite sensitive to outliers in the data. This has led some to consider subsets of data that meet a certain density threshold, but choosing the ``correct'' density threshold can be difficult.  Lesnick and Wright \cite{bifiltrations} consider all densities simultaneously, as a bifiltration of the data.  This yields a functor from  a poset $P \rightarrow Top$, where $P$ is the cross poset of two totally ordered posets. Again composing with the homology functor in chosen degree, we arrive at a functor from $P$ into $vec_\mathbb{F}$, a \emph{2-parameter persistence module. } One would hope that there is some equivalent theory of 2-parameter persistence modules as to that of 1-parameter. Alas, increasing the dimension of parameter space opens up a new can of worms. In particular, there are pointwise-finite-dimensional, indecomposable, 2-parameter persistence modules that are not interval modules (See Botnan et. al \cite{Botnan2022}).

As such, generalizing persistence diagrams to the multiparameter setting is non-trivial, and various workarounds have been proposed. Consideration has been made to generalizing the methods for vectorizing 1-parameter persistence modules to vectorize multi-parameter persistence modules. For example, Vipond \cite{MultiLandscapes} restricts a multi-parameter persistence module to lines in parameter space, and vectorizes the resulting family of 1-parameter persistence diagrams.

In a different direction that does not restrict to linear sub-posets,   Kim and M\'emoli \cite{KimMemoli} and 
Botnan, Oppermann, and Oudot \cite{signedbarcodes} have made progress in encoding  multi-parameter persistence modules into\emph{ signed barcodes}. We may represent these signed barcodes as series of signed points in Euclidean space, or more specifically, a these points exist in a particular polyhedron in Euclidean space. In unrelated work, Wagner et. al. \cite{mixup} use persistent homology tools to encode the mixture of two classes of data. The resulting persistence barcode is a set of triples of real numbers, which they call the \emph{mixup barcode}. These can also be viewed as points in a polyhedron in $\mathbb{R}^3$.

These two kinds of persistent diagrams are examples of the more general persistence diagrams we define in section 5; \emph{signed persistence diagrams on a pair $(X,A)$}, where $X$ is a polyhedron in Euclidean space and $A \subset X$ is a proper subset, composed of a finite union of sub-polyhedrons. 

More abstractly, we may think of persistence diagrams as arising from a random variable $Z$ from a probability space into a summary space of $d-$parameter persistence diagrams. The summary space of persistence diagrams is a metric space under the 1-Wasserstein distance, but not a vector space. Suppose we are given a collection $\{Z_i\}_{i=1}^N$ of random variables with the same distribution. We would like to have some way of representing a mean $\bar{Z}$ and show if $\bar{Z_i}$ converges to $\bar{Z}$, as in the setup by Bubenik \cite{Landscapes}.

 For this purpose, we make use of vectorizations of persistence diagrams, which embed them into a vector space. Multiple methods have been proposed for  persistence diagrams derived from 1-parameter persistence modules, with strong results: see Persistence Landscapes (Bubenik, \cite{Landscapes}),  Persistence Images, (Adams et. al.\cite{Adams}),  and Betti Curves (Umeda, \cite{umeda}).


We propose an embedding which maps signed persistence diagrams in arbitrary polyhedral pairs to sequences of real numbers in $\ell^1$. This method evaluates a sequence of functionals on persistence diagrams in order to embed them into sequence space. We note here that other works use various functional evaluation methods to embed 1-parameter persistence diagrams to lists of real numbers. Jose Parea, Liz Munch, and Firas Khasawneh \cite{MunchParea} use a template system of functionals on the plain to distinguish between specified groups of persistence diagrams for classifications. Furthermore, Atish Mitza and Ziga Virc \cite{atish} define functionals directly on the space of 1-parameter persistence diagrams of at most n-points, in order to map them to lists of real numbers with specified distortion functions. The method we propose employs a Schauder Basis of the order continuous dual space of signed persistence diagrams, and is general enough to be implementable on signed barcodes of arbitrary parameters, as well as variations of multiparameter persistence such as that of \emph{mixup barcodes} \cite{mixup}. Additionally, the Schauder Basis we propose is minimal in order to induce an embedding of the completion space of persistence diagrams, which consists of relative Radon measures on polyhedral pairs in Euclidean space.



We will begin in section 2, outlining various topics we think are prerequisite for definitions and proofs later in the paper. The last subsection, 2.4, is used in section 6 to expand earlier results to the space of relative Radon measures on a polyhedral pair. The material in this subsection inspired our method, but is not necessary for understanding the paper up to section 6.

In section 3, we define and explore properties of \emph{Nested Triangulations} (Definition \ref{def:nested triangulation}) on the polyhedral pair $(X,A)$. In Section 4, we use nested triangulations to define a family of functionals $\B$ on the pair $(X,A)$. We then prove that these functionals in $\B$ make up a \emph{Schauder Basis} of the vector space of compactly supported Lipschitz functionals on the pair $(X,A)$, denoted $Lip_c(X,A)$. That is, if $f \in Lip_c(X,A)$ is such a functional,  there is a unique sequence of scalar multiples of functionals in $\B$, the partial sums of which converge to $f$ (Theorem \ref{thrm: Schauder Basis}). The dual space of $Lip_c(X,A)$   in fact contains persistence diagrams on the pair $(X,A)$ \cite{PeterAlex}. This duality helps us use $\B$ to define an embedding $F_\B$ in section 5, mapping signed  persistence diagrams on the pair $(X,A)$ to sequences of real numbers in $\ell^1$.  In the case that $X \subset \mathbb{R}^2$, we may visualize this embedding of a persistence diagram $\alpha = \sum\limits_{i=0}^\infty x_i$ by decomposing the vector $F_\B(\alpha)$ into vectorizations of each individual point $x_i$. In the case that $\alpha$ is a persistence diagram of a 1-parameter persistence module,  then this collection of vectors may be viewed as a collection of color coded line segments in $\mathbb{R}^3$, with the line segment on point $x_i$ having color coded sections of length proportional to the entries of the vectorization of $x_i$.  (See figure \ref{fig:Persistence Mountainrange})
\begin{figure}[H]
    \begin{tikzcd}
        \includegraphics[scale=0.4, trim=30pt 10pt 50pt 15pt, clip]{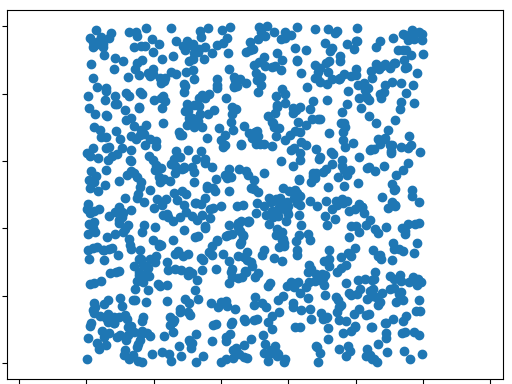}  & \includegraphics[scale=0.4, trim=7pt 10pt 80pt 12pt, clip]{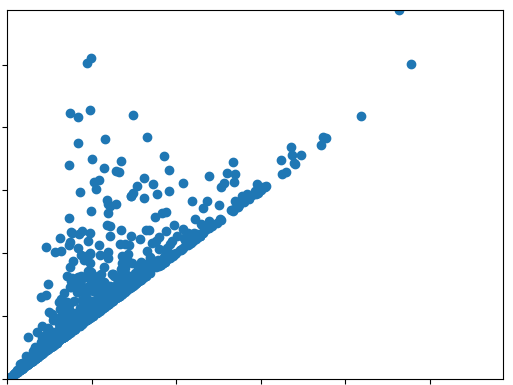}& \includegraphics[width=4cm, trim =550pt 100pt 150pt 50pt, clip]{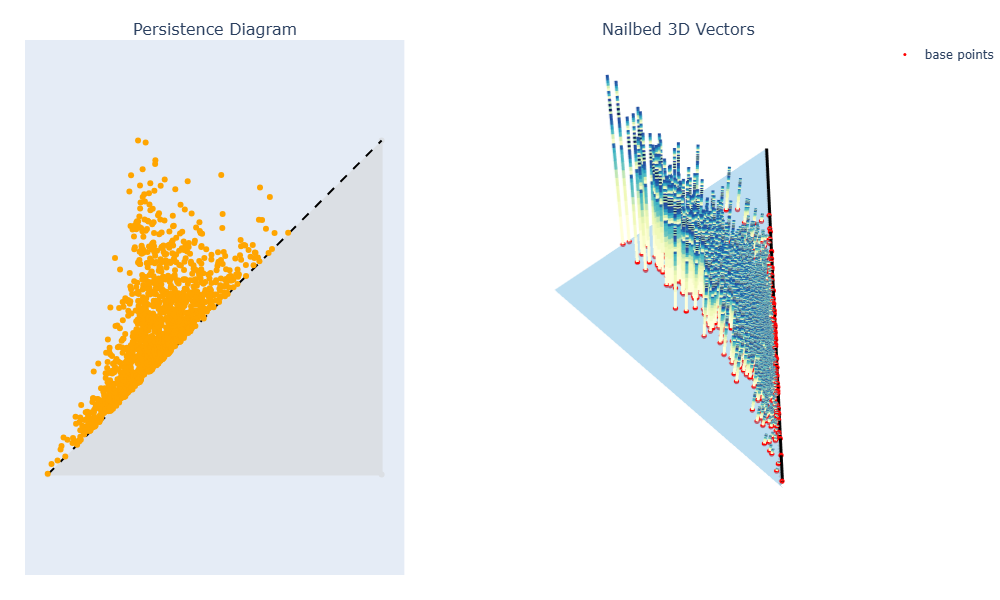}\\
    \end{tikzcd}
    \caption{\textbf{Left:} Point cloud $X \subset [0,1]^2$; \textbf{Middle:} Persistence diagram $\alpha$ in homology degree 1; \textbf{Right:} Visualization of $F_\B(\alpha)$}
    \label{fig:Persistence Mountainrange}
\end{figure}

Expanding to the 2-parameter setting, signed barcodes are represented by signed line segments in $\mathbb{R}^2$. Similar to the 1-parameter setting, we may vectorize each line segment individually. For the vectorization of line segment $x_i$, we stack sheets atop this line segment, parallel to the z-axis, with heights proportional to the entries of the vectorization of $x_i$. In this setting, line segments of negative multiplicity have sheets stacked in the negative z direction.

\begin{figure}[H]

    \begin{tikzcd}
        \includegraphics[scale=0.4, trim =80pt 45pt 120pt 50pt, clip]{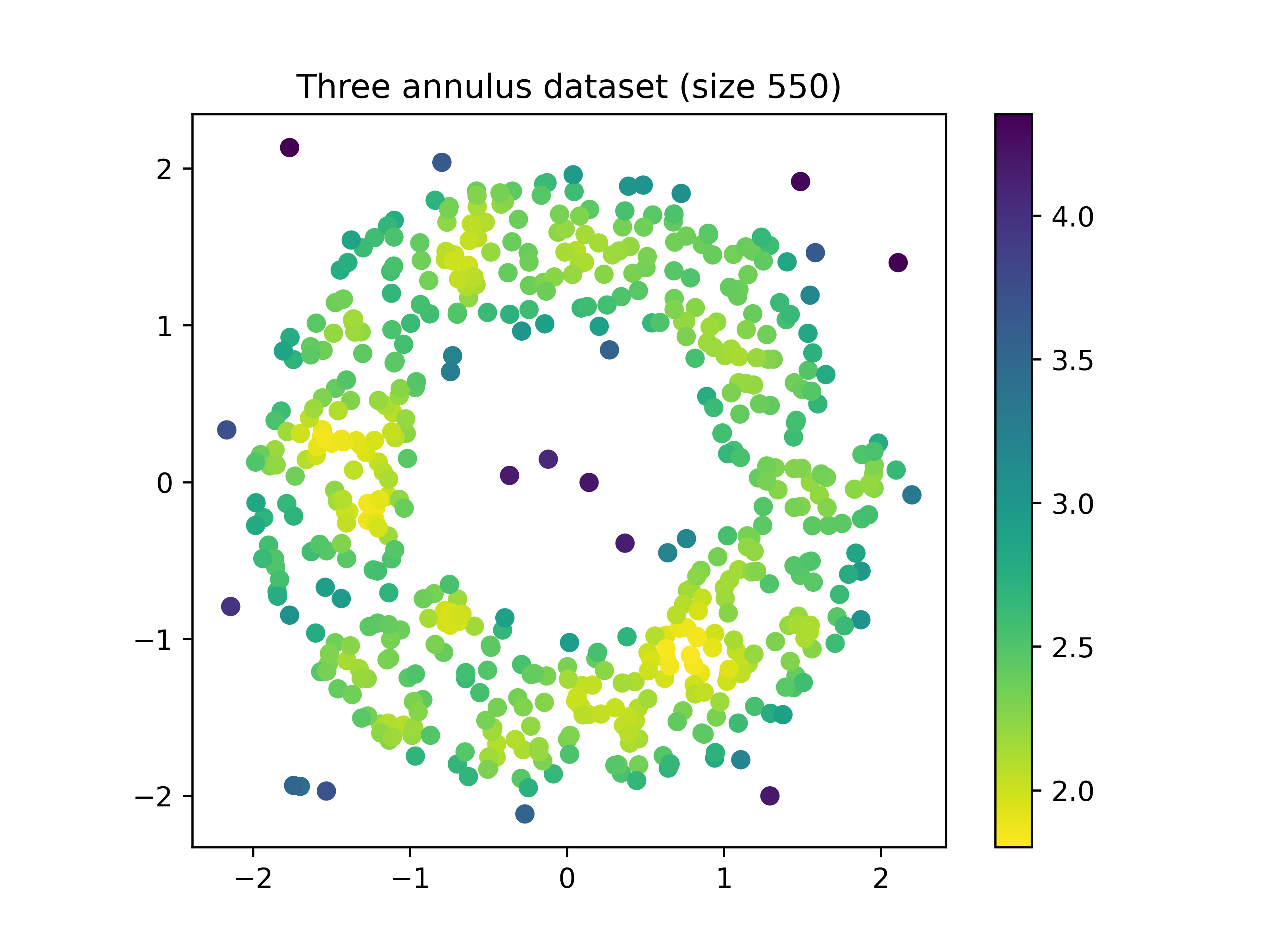}  & \includegraphics[trim=40pt 40pt 100pt 20pt, clip, scale=0.4]{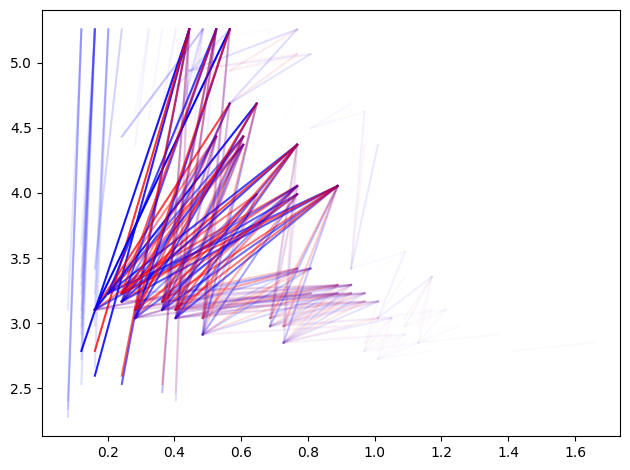}& \includegraphics[width =5cm]{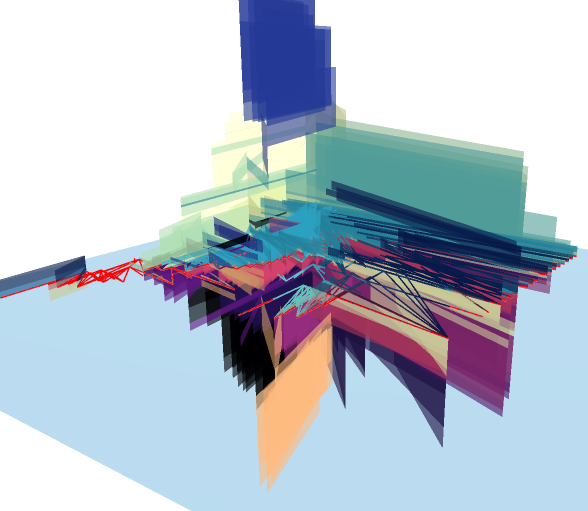}\\
    \end{tikzcd}
        \caption{\textbf{Left:} Point cloud $X \subset [0,1]^2$; \textbf{Middle:} $2-$ parameter signed barcode for the density-rips filtration,  generated by the Multipers package (Loiseaux and Schreiber, 2024) \cite{multipers}  \textbf{ Right:} Visualization of $F_\B(\alpha)$ in $\mathbb{R}^3$}

        \label{fig:2-parameter visuals pipeline}
\end{figure}

For an arbitrary Schauder Basis $\B$ of  $Lip_c(X,A)$ meeting certain conditions, we prove injectivity and stability of $F_\B$ as a map on persistence diagrams into vector space $\ell_1$ (Theorem \ref{thrm: schauder stability}). We then refine these results and gain a precise bound for the stability of $F_\B$ when $\B$ is derived from a Coxeter Freudenthal Kuhn nested triangulation using the method of section 4 (Theorem \ref{thrm:optimal bound}). We note that stability of $F_\B$  on persistence diagrams is a special case of stability on relative Radon measures, discussed in section 6. When $F_\B$ is considered as a map on relative Radon measures, we prove the Schauder Basis $\B$ formed as in section 4 is minimal for defining an embedding (Theorem \ref{thrm:minimality}).




\section{Background}

In this section, we review what we consider to be prerequisite knowledge for understanding various components of this paper. The last subsection, \emph{Relative Radon Measures}, is only necessary for extending our results in the final section of this paper, in which we extend our results to the completion space of persistence diagrams as we define them.

\subsection{Signed Barcodes as Persistence Diagrams in Polyhedral Pairs}\label{sec:2.1}

A common representation of 1-parameter persistence modules is persistence diagrams. These exist as a colellection, or formal sum, of points in the upper-left half of the plane, $\mathbb{R}^2_\leq = \{(x,y) \in \mathbb{R}^2 \ | \ x \leq y\}$. In recent research on variations of the 1-parameter persistence modules, including multiparameter persistence, it is useful to have summary tools for persistence modules that are not 1-parameter. Such generalizations of persistence diagrams may exist in higher dimensional relatives of the well-studied $\mathbb{R}^2_\leq$. We define the space of such persistence diagrams here.

 We say the space $X\subset \mathbb{R}^d$ is a \emph{polyhedron} if it is the finite intersection of half spaces in $\mathbb{R}^d$. If  $A \neq \emptyset$ is a finite union of sub-polyhedrons of $X$, then we say the pair $(X,A)$ is a \emph{polyhedral pair}. \\
 Let $(X,A)$ be such a polyhedral pair in $\mathbb{R}^d$. Let $D_+(X) =\{\sum\limits_{i=0}^{N\leq \infty}  x_i \ | x_i \in X \ \text{ and } \ \sum\limits_{i=0}^N d(x_i,A) < \infty\}$ denote the commutative monoid of (possibly infinite) formal sums of points in $X$, having finite cumulative distance to $A$.  \\

We will additionally denote the Grothendieck completion of $D_+(X)$ as $$D (X) = \{\sum\limits_{i=0}^{N\leq \infty} \pm x_i \ | \ \sum\limits_{i=0}^N d(x_i, A)< \infty\}$$ This is the abelian group of all signed formal sums of points in $X$, having finite cumulative distance to $A$. 

The \emph{space of persistence diagrams on $(X,A)$} is denoted $D(X,A)$ and is the abelian group of equivalence classes in $D(X)/D(A)$. We will denote elements of this space as $\alpha = \sum\limits_{i=0}^{N \leq \infty} \pm x_i$, when we mean an equivalence class $\alpha$ with representative $\sum\limits_{i=0}^{N \leq \infty} \pm x_i$. We will denote the equivalence class of persistence diagrams in $A$ simply by $A$. That is, $\sum\limits_{i=0}^\infty  \pm x_i = \left( \sum\limits_{i=0}^\infty  \pm x_i\right) +A$ . In some situations, it is useful to separate the positive terms of a persistence diagram $\alpha$ from its negative terms. In such context, we will denote this separation by $\alpha = (\alpha^+, \alpha^-)$.\\

 A persistence diagram $\alpha =(\alpha^+, \alpha^-)$ may be  viewed as a signed measure $\alpha = \sum\limits_{x \in \alpha^+} \delta_x - \sum\limits_{x \in \alpha^-} \delta_x$, where $\delta_{x}$ is the Dirac measure at $x$. We note that persistent diagrams on the polyhedral pair $(X,A)$  are a special case of the setting defined by Che et. al.  \cite{Che2024} and Bubenik and Elchesen \cite{PeterAlex}.  \\

\begin{example}

  A persistence module $M$ is a functor from a poset category $(P, \leq)$ into the category of vector spaces over a chosen field $\mathbb{F}$. For interval $I \subset P$, we may define the \emph{interval module} $\mathbb{F}_I: P \rightarrow vec_\mathbb{F}$ such that $\mathbb{F}_I(t) =\mathbb{F}$ iff $t \in I$, and $\mathbb{F}_I(t)=0$ if $t \notin I$.  Crawley and Boevey \cite{Crawley} have shown that in the case that $P$ is totally ordered,  any pointwise finite dimensional, persistence module $M$ on $P$ is decomposable into interval modules and that this decomposition is unique up to isomorphism. $M \simeq \bigoplus\limits_{\lambda} \mathbb{F}_{I_\lambda}$ for some collection of intervals $\{I_\lambda\}_{\lambda \in \Lambda}$. The collection $\{I_\lambda\}_{\lambda \in \Lambda}$ is called the \emph{persistence barcode} of $M$. 
  
     The rank invariant of a 1-parameter persistence module  $M$ is a function $Rk(M): \mathbb{R}^{2}_\leq \rightarrow \mathbb{Z}_+$ assigning each pair $s \leq t \in \mathbb{R}$ the rank of the map $M(s \leq t)$. Carlsson and  Zomorodian \cite{Carlsson2009} have shown that the rank invariant on 1-parameter persistence modules is complete, i.e. it determines the isomorphism class of a persistence module. The rank invariant is equivalent to the barcode through M\"obius inversion. More specifically, the rank $Rk(M)(s,t)$ is equal to the number of intervals in the barcode of $M$ containing both $s$ and $t$.  \\

 For an interval $\mathbb{F}_{I_\lambda}$ module on interval  $I_\lambda$ with $\inf (I_\lambda)=a_\lambda$ and $\sup (I_\lambda)=b_\lambda$, the M\"obius inversion of the Rank function on module $M_{I_\lambda}$ may be encoded as the function $PD_{I_\lambda} : \mathbb{R}^2_\leq \rightarrow \mathbb{Z}_+$, which is the indicator function on the point $(a_\lambda,b_\lambda) \in \mathbb{R}^2_\leq$. The \emph{persistence diagram of $M$} is the function $PD_M: \mathbb{R}^2_\leq \rightarrow \mathbb{Z}_+$, such that $PD_M:= \sum\limits_{\lambda \in \Lambda} PD_{I_\lambda}$. \\

 We may plot the point $(a_\lambda, b_\lambda)$ in $\mathbb{R}^2_\leq$ for any interval module $\mathbb{F}_{I_\lambda}$, and thus plot the persistence diagram in the space $\mathbb{R}^2_\leq $ . Persistence diagrams inherit a commutative monoid structure from the commutative monoid structure of p.f.d. persistence modules. In this setting, these  persistence diagrams make up the space $D_+(\mathbb{R}^2_\leq)$ as we've defined it above. Furthermore,  intervals $I_\lambda$ such that $a_\lambda = b_\lambda$ are considered to be ephemeral, and often disregarded. Equivalence classes of persistence diagrams under this relation make up the quotient monoid, $D_+(\mathbb{R}^2_\leq, \Delta)$.
\end{example}

\begin{example}\label{example:Signed Barcodes}

 One would hope that there is some equivalent theory of 2-parameter persistence modules as to that of 1-parameter. In particular, it would be nice if there was decomposability of 2-parameter persistence modules into rectangle modules. Alas, this is not the case. In fact, there are 2-parameter persistence modules that do not even decompose into interval modules, and the rank invariant is not complete on 2-parameter persistence mdoules. (For further results on subclasses of multi-parameter modules on which the rank invariant is complete, see Botnan et. al. \cite{Botnan2022}. )
    
    We thus look to alternative theories to study multiparameter persistence modules, particularly those on which the rank invariant is not complete. Kim and M\'emoli \cite{KimMemoli} define the \emph{generalized rank invariant} on multi-parameter persistence modules,  and show that it is complete on a subclass of interval decomposable modules, given some finiteness conditions on the indexing poset. (A greater extent of this completeness is studied in Clause, Kim and Me\'moli \cite{IncompletenessofGRI}).

    For interval $I \subset P$, and persistence module $M:P \rightarrow vec$, they define $RK_I(M) = Rk\left( \lim\limits_{\leftarrow} M|_I \rightarrow \lim\limits_{\rightarrow} M|_I\right)$. 

For a collection $\mathcal{I}$ of intervals in $P$, and persistence module $M$, the \emph{generalized rank invariant of $M$ (relative to $\mathcal{I}$)} is the map $RK_\mathcal{I}: \mathcal{I} \rightarrow \mathbb{Z}_+ $ mapping $I \mapsto RK_I(M)$. \\

When a module is interval decomposable, we refer to the set of intervals defining its decomposition as the \emph{barcode} of the module. Kim and M\'emoli \cite{KimMemoli} show that whenever $M$ is interval decomposable on the specified subclass $\mathcal{I}$ of intervals, the multiplicities of intervals in the barcode of $M$ are given by M\"obius inversion of the generalized rank invariant . This method can be applied to the generalized rank invariant of a persistence module $M$ which is not interval decomposable, but the multiplicities emerging from M\"obius inversion in this case may be negative.  \\

Botnan, Oppermann, and Oudot, \cite{signedbarcodes}, use the generalized rank invariant in such a setting to define \emph{signed barcodes}. They consider the case of the generalized rank invariant over a set of intervals $\mathcal{I}$ that is either the set of half-open rectangles in $\mathbb{R}^d$, or hook modules in $\mathbb{R}^d$. They show that for finitely presented, pointwise-finite-dimensional persistence modules on $\mathbb{R}^d$,  there exists a unique pair of disjoint sets (with multiplicity) $R,S$ of intervals in $\mathcal{I}$, such that the generalized rank invariant is decomposed through $(R,S)$; that is $$Rk_\mathcal{I} (M) = Rk_\mathcal{I}\left(\bigoplus\limits_{I\in R} M_{I}\right) - Rk_\mathcal{I}\left(\bigoplus\limits_{I \in S} M_{I}\right)$$

We represent signed barcodes of this form as signed persistence diagrams on a polyhedral pair. For $d-$parameter persistence module $M$, let $\mathbb{R}^{2d}_\leq = \{x \in \mathbb{R}^{2d} \ | \ x_i \leq x_{i+1} \ \forall i \leq d\}$. If $[a,b)= \{x | \ a \leq x < b\}$ is a rectangle in $\mathbb{R}^d$, we encode the M\"obius inversion of $Rk(M_{I[a,b)})$ as the function $PD_{[a,b)} : \mathbb{R}_\leq^{2d}\rightarrow \mathbb{Z}_+$ which is the indicator function on the point $(a,b) \in \mathbb{R}_\leq^{2d}$, where $(a,b) := (a_1, a_2,... a_d, b_1, ...b_d)$. Building on this, the M\"obius inversion of $Rk_\mathcal{I}(M)$ is encoded as the function $PD_M: \mathbb{R}^{2d}_\leq \rightarrow \mathbb{Z}$, which is the sum of signed indicator functions on points $(a,b)$, with positive sign if $[a,b) \in R$ and negative if $[a,b) \in S$. 

Similarly, if one prefers to work with hook modules for some applications, we may also define a signed persistence diagram of $M$ through a pair of sets of hook modules $R,S$. Here, the hook $[a,b[:= \{x \ | \ a \leq x \ \text{ and } x \ngeq b\}$ is mapped to the point $(a,b) \in \mathbb{R}^{2d}_\leq$. 



  In either construction, persistence modules are represented as formal sums of signed points in  $\mathbb{R}^{2d}_\leq $, which is the intersection of half-spaces in $\mathbb{R}^{2d}$ defined by $x_i \leq x_{i+1}$ for $i \leq d$.  Consider the subset of $\mathbb{R}^{2d}_\leq$ consisting of the  bounding hyperplanes of the form $x_i=x_{i+d}$. Points on these hyperplanes represent "flat" rectangles in the signed barcode, and we may consider these features to be ephemeral. We will denote the space containing these points as $\Delta^{d}$. Then equivalence classes of these signed diagrams, modulo ephemeral diagrams, make up the space $D(\mathbb{R}^{2d}_\leq, \Delta^{d})$. 
\end{example}

\begin{example}[Mixup Barcodes]\label{example:mixup}
    From a different perspective, Wagner et. al. \cite{mixup} study the mixup of a pair of classes of data. In so doing, they compare two 1-parameter persistence modules, $M_R$ generated by a data class $R$, and  $M_{R \cup B}$ generated by the union of the two data classes $R\cup B$. The inclusion $R \hookrightarrow R \cup B$ induces a map from $M_R \rightarrow M_{R\cup B}$. They encode information from these persistence modules and the map between them as triples $(b, d',d) \ $ where $ \ b \leq d' \leq d \in \mathbb{R}$. Here $b$ is the birth time of a feature, $d'$ is its death in $M_{R \cup B}$, and $d$ is its death in $M_R$. The ``mixup'' of a bar is defined as $d-d'$, and large mixup is considered an indicator that data points in the class $B$ are in a sense, surrounded by points in class $R$.

    For example, consider two classes of data as illustrated in Figure \ref{fig: mixup data}. One class $R$ is of points arranged in a circle or radius $1$ in $\mathbb{R}^2$, and a class $B$ a single points in the center of the circle (see Figure \ref{fig: mixup data}). When considering the inclusion $R \hookrightarrow R \cup B$, we may determine the vietoris-rips filtration of both $R$ and $R \cup B$. The persistence barcode of module $H_1(VR(R))$ has one significant feature, born at approximately time $0.2$, dying at $\sqrt{3}$. Now considering the map induced by inclusion $H_1(VR(R)) \rightarrow H_1(VR(R \cup B))$, the bar $[0.2, \sqrt{3}]$ is paired with the bar $[0.2, 1]$. Thus, the corresponding 3-bar of the mixup barcode is $[0.2, 1, \sqrt{3}]$, where the mixup of this feature is $\sqrt{3}-1$.

The triples that define the mixup barcode can be viewed as points in a subspace of 3 dimensional Euclidean space, $\mathbb{R}^3_{\leq, \leq}=\{(x,y,z) \ | \ x \leq y \leq z\}$. Again $\mathbb{R}^3_\leq$ is equivalently defined as the intersection of the half-spaces $x \leq y $ and $y \leq z$, and is thus a polyhedron. One may consider points with no mixup to be ephemeral. These are points that exist on the plane $y=z$, which we will denote $\Delta^M$.  Thus the mixup barcodes make up the space $D_+(\mathbb{R}^3_{\leq, \leq}, \Delta^M)$.

\begin{figure}[H]
    \centering
    \includegraphics[width=0.5\linewidth]{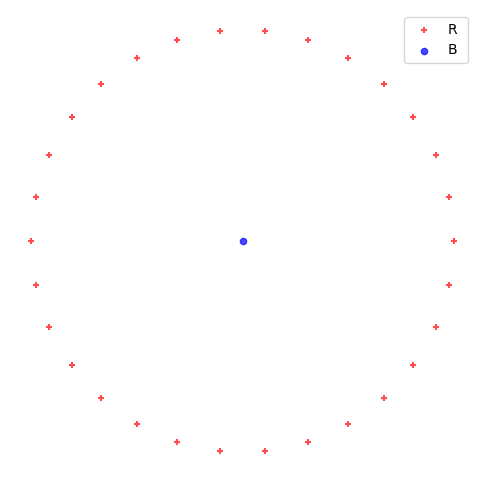}\includegraphics[width=0.5\linewidth]{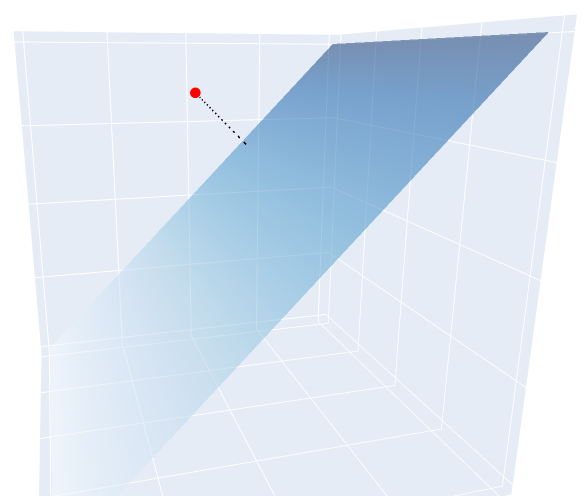}
    \caption{\textbf{Left:Two classes of data, compared with a mixup barcode}: \textbf{Right} Mixup Diagram of this pair of data. The dotted line represents the ``mixup'' of the feature}
    \label{fig: mixup data}
\end{figure}

\end{example}

For polyhedral pair $(X,A)$, we will endow the space $D(X,A)$  with the 1-Wasserstein distance, a metric induced by a metric on $\mathbb{R}^d$, such as Eclidean or $\ell_1$ distance.

 We start by assuming $\alpha,\beta \in D_+(X,A)$. That is, all terms of $\alpha$ and $\beta$ have positive sign. \\

Choose representatives of the equivalence classes $\alpha$ and $\beta$ having infinitely many terms equal to $A$; that is, $\alpha = \sum\limits_{i=0}^\infty x_i $ , $\beta =  \sum\limits_{i=0}^\infty y_i $ such that infinitely many $x_i$ and infinitely many $y_i $ equal $A$. \\

Let $\sigma : \mathbb{N} \rightarrow \mathbb{N}$ be a bijection. The \emph{Cost} of $\sigma$ relative to $\alpha$ and $\beta$ is 

$$Cost(\sigma) := \sum\limits_{i=0}^\infty d(x_i, y_{\sigma(i)})$$

We refer to $\sigma$ as a \emph{partial matching} of $\alpha$ and $\beta$. Note that $d(x_i, y_{\sigma(i)}) =0$ if $x_i = y_{\sigma(i)} = A$\\

Note that this is in fact always finite by the triangle inequality. 
\begin{equation*}
    \begin{split}
        \sum\limits_{i=0}^\infty d(x_i, y_{\sigma(i)})&\leq \sum\limits_{i=0}^\infty d(x_i,A) + d(y_{\sigma(i)}, A)\\
        &= \sum\limits_{i=0}^\infty d(x_i,A) +  \sum\limits_{i=0}^\infty d(y_{\sigma(i)},A)
    \end{split}
\end{equation*}

 The \emph{1-Wasserstein Distance} between persistence diagrams $\alpha$ and $\beta$ is the infemum of costs of all partial matchings $\sigma$. 

$$W_1(\alpha, \beta):= \inf\limits_{\sigma} Cost(\sigma)$$

Now consider the case that $\alpha, \beta$ do not consist solely of positive terms. Then $\alpha = (\alpha^+, \alpha^-)$ and $\beta = (\beta^+, \beta^-)$.Then define

$$W_1(\alpha, \beta) := W_1(\alpha^+ +\beta^-, \beta^+ + \alpha^-)$$

This makes $D(X,A)$ into a metric space, and in fact, a normed vector space, where $||\alpha||= W_1(\alpha^+, \alpha^-)$. This norm is equal to the Kantorvich-Rubinstein norm. 

 \subsection{Piece-wise Linear Triangulations}

If $T$ is a simplicial complex, and $S\subset T$ is a subcomplex, we refer to the pair $(T,S)$ as a \emph{simplicial pair}.

Let $T$ be a simplicial complex and $f:|T| \rightarrow W$ be a map from the geometric realization of $T$ into the $\mathbb{R}$ vector space $W$.   We say $f$ is \emph{piece-wise linear (p.l.)} on $T$ if $f$ is linear with respect to Barycentric coordinates of $T$. That is,  $\forall \sigma =\langle v_0, ...v_p\rangle  \in T$ and $x \in |\sigma| $ with Barycentric coordinates $x= \sum\limits_{i=0}^p a_i \cdot v_i$, the value of $f$ at $x$ is the Barycentric average of $f$ on vertices of $\sigma$; 

    $$f(x) = \sum\limits_{i=0}^p a_i \cdot f(v_i)$$

    For polyhedral pair $(X,A)$, a \emph{piece-wise linear triangulation of $(X,A)$} is a simplicial pair $(T,S)$ along with a homeomorphism of pairs $h:(|T|, |S|) \rightarrow (X,A)$, which is p.l. on $T$. \\

\begin{example}
    Let $D$ be a set of points uniformly sampled from the unit disk in $\mathbb{R}^3$, and specify a singular point at $0$. Let $T$ denote the Delaunay triangulation of $D$. Let $h: |T|\rightarrow \mathbb{R}^3$  be the p.l. extension of the identity map on points. Let $X$ be the image of $h$. Then $h$, when thought of as a map of pairs $h:(|T|,0)\rightarrow (X,\{0\})$ is a p.l. triangulation. (Figure \ref{Triangulated Sphere}).

\end{example}\begin{figure}[H]
        \centering
        \includegraphics[width=0.4\linewidth]{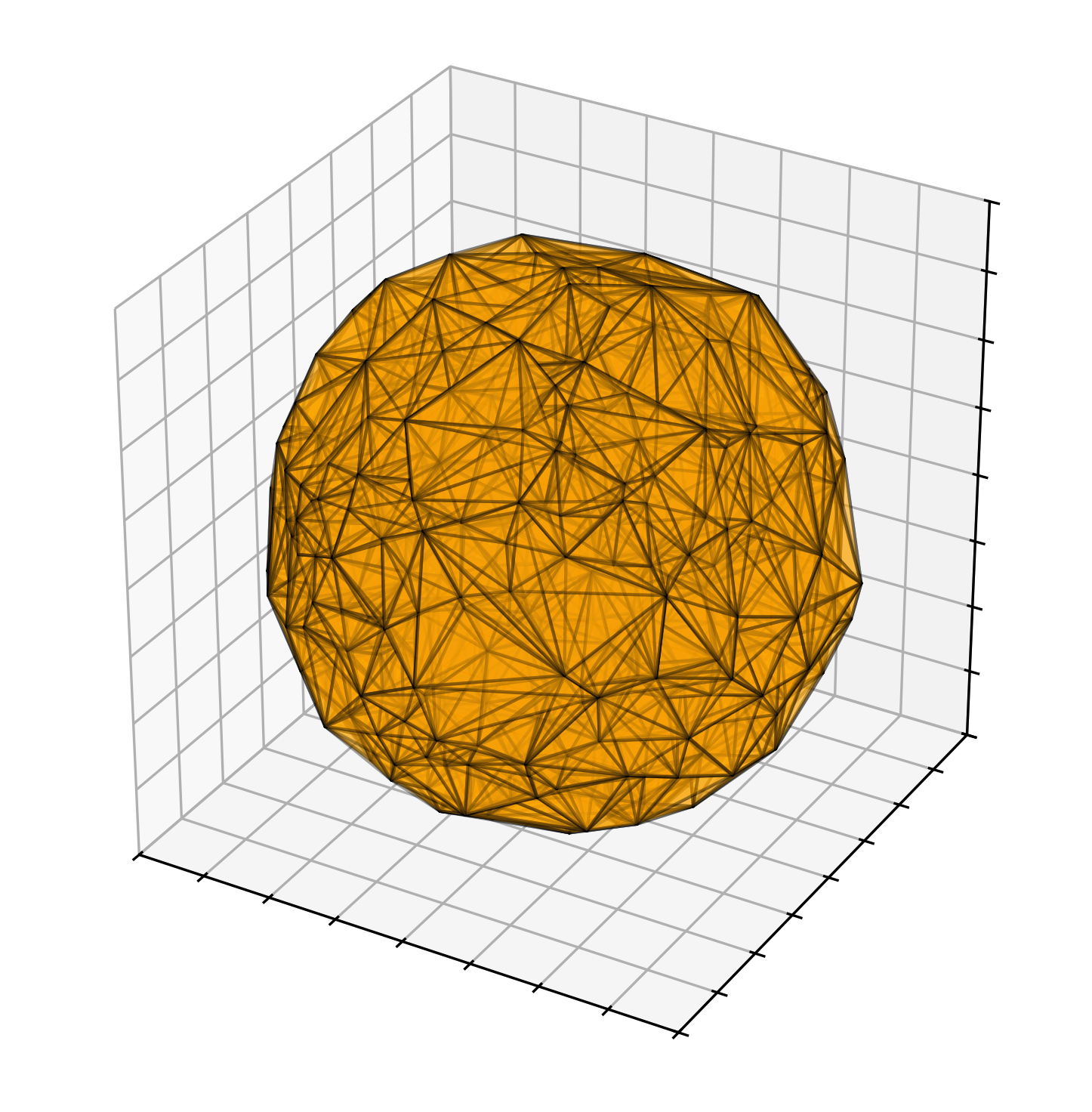}
        \caption{A triangulated approximation of a sphere, $X$, with closed subset $A $ consisting of a single point at the center}
        \label{Triangulated Sphere}
    \end{figure}

 For the remainder of this paper, we shall abuse notation to suppress the homeomorphism $h$. That is, when $h: (|T|,|S|) \rightarrow (X,A)$ is a p.l. homeomorphism making $(T,S)$ a p.l. triangulation of $(X,A)$, we will not explicitly refer to $h$, and instead simply say, ``$(T,S)$ is a triangulation of $(X,A)$''. Furthermore, for $\sigma \in T$ we will  simply say `` $x \in \sigma$'' instead of ``$x \in h(|\sigma|)$''.

\begin{remark}\label{Remark: finite simplices}
    P.L. triangulations on polyhedral pairs have the useful property that the triangulation is in a sense, locally finite.  That is, for $T$ a triangulation of $X$ and $C \subset X$ compact,there exists at most finitely many simplices of $T$ which intersect $C$. By extension, since $X$ is locally compact, for any point $x \in X$, there exists finitely many simplices of $T$ which contain $X$. 


\end{remark}

Most of our results will apply to sequences of triangulations on an arbitrary polyhedral pair.  However, we will later optimize some of our results for sequences of a specific kind of triangulation, which we define now. 

   \begin{definition}
       
 Let $\Upsilon \subset \{ (i,j) | \ 1 \leq i < j \leq d\}$ be a set of relations on indices of $\mathbb{R}^d$. Let $X \subset \mathbb{R}^d$ be the subset of $\mathbb{R}^d$ such that $X = \{x \ | \ x_i \leq x_j \ \forall (i,j) \in \Upsilon\}$.  Furthermore, let  $\Upsilon ' \subset \Upsilon$ be a nonempty subset of relations. Let $A = \{x \in \mathbb{R}^d \ | \ \exists (i,j) \in \Upsilon ' \ s.t.  \ x_i = x_j \}$. The \emph{Coxeter-Freudenthal-Kuhn (CFK) triangulation at scale 1} on $(X,A)$  is the triangulation $(T,S)$ formed by the intersection of all hyperplanes $x_i = x_j + \mathbb{Z}$ for all $i \leq j \leq d$, along with all hyperplanes $x_i =\mathbb{Z}\ ; \forall i \leq d$. 

    Alternatively, the \emph{CFK-Triangulation at scale $ 1$} is the triangulation of simplices of the following form. Let $v_0 \in \cdot  \mathbb{Z}^d \cap X$, and let $\pi \in S_d$ be a permutation on $d$. Let $1 \leq p \leq d$, and let $ v_i: = v_0 + \sum\limits_{j=1}^i \cdot e_{\pi(j)}$ . Provided $v_i \in X $ for all $i$, then $\sigma = \langle v_0, ...v_p\rangle $ is a simplex of $T$.  
    
    For $c>0$, the \emph{CFK-Triangulation at scale $ c$} is the equivalent triangulation scaled by $c$, formed by hyperplanes $x_i = x_j + c\cdot \mathbb{Z}_+$ for all $i \leq j \leq d$, along with all hyperplanes $x_i = c \cdot \mathbb{Z}$. Alternatively, the \emph{CFK-Triangulation at scale $ c$} is the triangulation of simplices of the following form. Let $v_0 \in c\cdot  \mathbb{Z}^d \cap X$, and let $\pi \in S_d$. Let $1 \leq p \leq d$, and let $ v_i: = v_0 + \sum\limits_{j=1}^i c\cdot e_{\pi(j)}$ . Provided $v_i \in X $ for all $i$, then $\sigma = \langle v_0, ...v_p\rangle $ is a simplex of $T$.  

   \end{definition}
We note a useful property of a CFK-triangulaion at scale $1$. Namely,  each $d$-simplex of $T$ is congruent to the standard $d$-simplex, $\sigma = \langle v_0, v_1, ...v_d\rangle $, where $v_0 =0$ and $v_p = \sum\limits_{i=1}^p e_i$ for $p >0$. Similarly, the simplices of the CFK triangulation at scale $c$ are congruent to the standard $d-$simplex scaled by $c$. The construction and properties of CFK triangulations are explored further by Boissonnat et. al.  \cite{CFK}

\begin{example}\label{example:CFK triangulation} Let $ \mathbb{R}^2_\leq = \{(x,y) \in \mathbb{R}^2 \ | \ x \leq y\}$ be the space in which 1-parameter persistence diagrams are represented. Furthermore, let  $\Delta:= \{(x,x)\ | \ x \in \mathbb{R}\}$ be the set containing points derived from instantaneously lived features in a persistence module. We consider these points to be \emph{ephemeral}. A CFK triangulation of polyhedral pair $(\mathbb{R}^2_\leq, \Delta)$ is formed by lines of the form $y=\mathbb{Z}$, $x = \mathbb{Z}$, and $y=x+\mathbb{Z}$. 

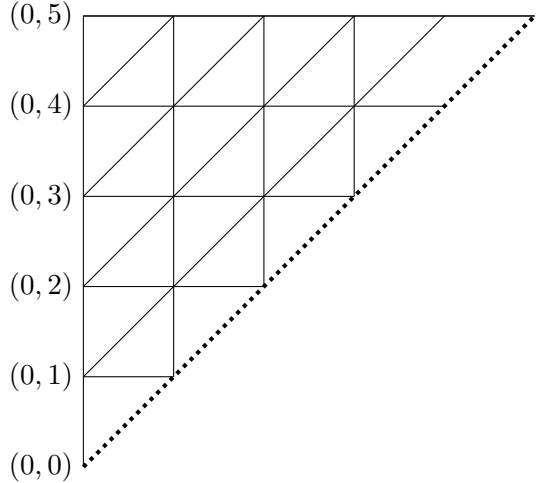
\begin{figure}[H]\label{fig: CFK triangulations}

\begin{tikzpicture}[scale=1.2]
            \draw[dotted, line width=0.55mm] (0,0)--(5,5);
            \draw (5,5)--(0,5)--(0,0);
            \draw (0,2)--(2,2);
            \draw (0,1)--(1,1);
            \draw (0,3)--(3,3);
            \draw (0,4)--(4,4);
            \draw (0,5) --(5,5);
            \draw (1,1)--(1,5);
            \draw (2,2)--(2,5);
            \draw (3,3)--(3,5);
            \draw (0,1) -- (4,5);
            \draw (0,2) -- (3,5);
            \draw (0,3) -- (2,5);
            \draw (0,4) -- (1,5);
            \node [black, left] at (0,0){$(0,0)$};
            \node [black, left] at (0,1){$(0,1)$};
            \node [black, left] at (0,2){$(0,2)$};
            \node [black, left] at (0,3){$(0,3)$};
            \node [black, left] at (0,4){$(0,4)$};
            \node [black, left] at (0,5){$(0,5)$};
\end{tikzpicture}

\caption{ The Coxeter Freudenthal Kuhn triangulation at scale 1 on polyhedral pair $(\mathbb{R}^2_\leq, \Delta)$}
\end{figure}

\end{example}

\begin{example}\label{example: CFK triangulation R3}
    Leaning in the direction of mhigher-dimensional persistence, let us consider mixup barcodes \cite{mixup}, and the polyhedral pair they occupy, $(\mathbb{R}^3_{\leq, \leq}, \Delta^M)$. The CFK triangulation at scale 1 of polyhedral pair $(\mathbb{R}^3_{\leq, \leq}, \Delta^M)$ is formed by the hyperplanes of the form $x_i = \mathbb{Z}$ for $i\in \{1,2,3\}$, as well as hyperplanes of the form $x_2=x_1 + \mathbb{Z}$, $x_3 = x_1+\mathbb{Z}$, and $x_3 = x_2 + \mathbb{Z}$.

    \begin{figure}[H]
        \centering
        \includegraphics[width=0.7\linewidth]{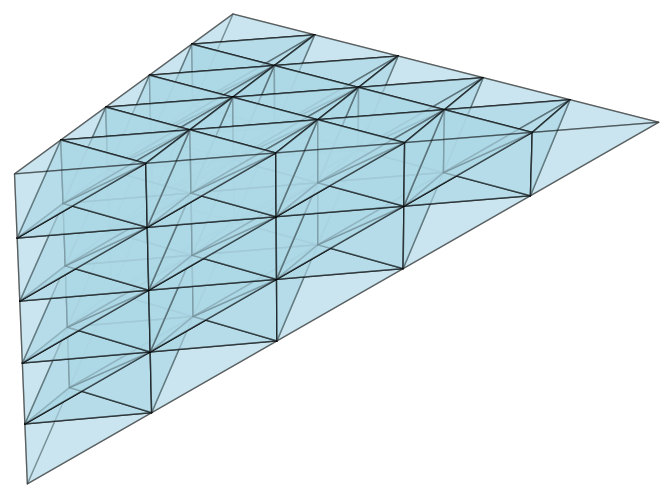}
        \caption{ A CFK triangulation of polyhedral pair $(\mathbb{R}^3_{\leq, \leq}, \Delta^M)$}
        \label{mixup CFK triangulation}
    \end{figure}
\end{example}

\subsection{Schauder Basis}
The key tool of this project is the notion of a Schauder Basis on a vector space. This is a looser notion than that of a more traditional Hamel basis, and requires some additional structure to be supposed on the vector space itself. 

   Let $(W, \Omega)$ be a topological vector space over the field $\mathbb{F}$. A \emph{Schauder basis} of $W$ is a countable subset  $\B=\{e_i\}_{i=0}^\infty \subset W$, such that for every $w\in W$, there exists a unique sequence $\{a_i\}_{i=0}^\infty $ of scalars in $\mathbb{F}$ such that 

    $$\sum\limits_{i=0}^\infty a_i \cdot e_i =v$$
    where convergence is taken to mean convergence of the partial sums in the topology $\Omega$. \\

\begin{example}
    Consider the $\mathbb{R}$ vector space $\ell^p$ for $p < \infty$. Let $\B$ be the set of standard unit vectors $e_i$ for $i \geq 0$. We illustrate below that $\B$ is a Schauder basis of $\ell^p$. 

    Let $a= (a_i)_{i=0}^\infty$ be a sequence in $\ell^p$. Let $s^n:=\sum\limits_{i=0}^n a_i \cdot e_i $.

    Note that $||a-s^n||_p^p= \sum\limits_{i=n+1}^\infty a_i^p$. But since $a \in \ell^p$, we know that $\sum\limits_{i=0}^\infty a_i^p < \infty$, thus $||a-s^n||_p^p \rightarrow 0$ and hence $s^n \rightarrow a$. Furthermore, the coefficients $a_i$ of $e_i$ are unique in this summation. Hence $\B$ is indeed a Schauder Basis of $\ell^p$. 
    
\end{example}

\subsection{Relative Radon Measures}

In this section, we will define a more general setting than that of signed persistence diagrams on a polyhedral pair. We will make use of this material in section 6 for extending results of section 5.

 Let $X\subset \mathbb{R}^d$ be a polyhedron. For functional $f:X \rightarrow \mathbb{R}$, and (signed or unsigned) measure $\alpha $ on $X$, we use notation $\alpha(f):= \int_{X} f \ d \alpha$. For a signed measure $\alpha$, we denote $\alpha$ by it's positive and negative components as $\alpha = (\alpha^+ , \alpha^-)$.

Let $\mathcal{B}^+(X) $ be the commutative monoid of unsigned Borel measures on $X$. If $A \subset X$, let $\mathcal{B}^+(X,A) := \mathcal{B}^+(X)/ \mathcal{B}^+(A)$. 

We now define a sub-monoid of these relative Borel measures. For measures $\alpha$ and Borel set $U$, let $\alpha_U$ denote the measure defined by $\alpha_U(E) = \alpha(E \cap U)$ for all Borel sets $E \subset X$. Let $d_A: X \rightarrow \mathbb{R}_{\geq 0} $ mapping $x \mapsto d(x,A)$. Then define the sub-monoid 

$$\hat{\mathcal{M}}^+_1(X,A) = \{\alpha \in \mathcal{B}^+(X,A) \ | \ \alpha \text{ is tight }\ \text{ and } \forall x \in X , \exists \text{ neighborhood } U \text{ with } \alpha_U(d_A)< \infty\}$$


Let $\hat{\mathcal{M}}_1(X,A)$ be the Grothendieck completion of $\hat{\mathcal{M}}_1^+(X,A)$. Note the persistence diagrams on polyhedral pair $(X,A)$ , $D(X,A)$, make up a subgroup of $\hat{\mathcal{M}}_1(X,A)$.

There is a generalization of the 1-Wasserstein distance to $\hat{\mathcal{M}}_1(X,A)$. To define this,  let $p_1, p_2: X^2 \rightarrow X$ be projections onto the first and second components respectively. Let $\alpha, \beta \in \hat{\mathcal{M}}^+_1(X,A)$  and  $\pi \in \mathcal{B}^+(X^2, A^2)$. Then $\pi$ is a \emph{coupling} of $\alpha$ and $\beta$ iff $(p_1)_*(\pi)=\alpha$ and $(p_2)_*(\pi)=\beta$. Let $\Pi(\alpha,\beta)$ denote the set of all such couplings of $\alpha$ and $\beta$.

An example of a coupling is the partial matchings of persistence diagrams used in defining the Wasserstein distance on $D_+(X,A)$.

    Let $\bar{d}= d \wedge (d_A \oplus d_A): X^2 \rightarrow \mathbb{R}_{\geq 0}$ be the minimum of the distance between two points, and the sum of the distances of points to $A$.

    We may define the (relative) 1-Wasserstein distance between relative Radon measures $\alpha, \beta$ as 

    $$W_1(\alpha, \beta) = \inf\limits_{\pi \in \Pi(\alpha, \beta)} \pi(\bar{d})$$

    We may extend this to signed persistence measures $\hat{\mathcal{M}}_1(X,A)$ where $\alpha$ and $\beta$ decompose into positive and negative components by the following.

    $$W_1(\alpha, \beta) := W_1(\alpha^+ + \beta^-, \beta^+ + \alpha^-)$$

For some proofs, we will make use of the following lemma regarding couplings. Let $Lip(X,A)$ denote the group of Lipschitz functionals on $X$ which are $0$ on $A$.  

\begin{lemma}\cite{PeterAlex}(Lemma 6.1d)\label{lem: couplings}
    If $\pi$ is a coupling of $\alpha, \beta \in \hat{\mathcal{M}}_1^+(X,A)$, and $f,g \in Lip(X,A)$, then $\pi(f \oplus g) = \alpha(f) + \beta(g)$.
\end{lemma}

We lastly include the theorem that inspired this project, outlining the duality of function space with persistence diagrams. Let $Lip_c(X,A)\subset Lip(X,A)$ denote the subgroup of compactly supported Lipschitz functionals.  $\hat{\mathcal{M}}_1(X,A)$, which contains signed persistence diagrams on $(X,A)$, makes up the sequentially order continuous dual of $Lip_c(X,A)$. 

\begin{theorem} \cite{PeterAlex}(Theorem 5.9)
    Let $(X,A)$ be a metric pair. Assume that $X$ is locally compact. Then $Lip_c(X,A)$ is the sequentially order continuous dual of $\hat{\mathcal{M}}_1(X,A)$. 
\end{theorem}

\section{Nested Triangulations}

In this section, we define and explore properties of sequences of nested p.l. triangulations. We use these nested triangulations to construct Schauder bases of Lipschitz functionals in section 4, and we use these Schauder Bases to define vectorizations in section 5.

 Recall that we refer to a map $h:|T|\rightarrow X$ to be p.l. if $h$ is consistent with Barycentric coordinates. That is, if $x \in |\sigma|$ and $x = \sum\limits_{i=0}^p a_i \cdot v_i$ for vertices $v_i$ of $\sigma$, then $h(x) = \sum\limits_{i=0}^p a_i \cdot h(v_i)$. 

Much of this section will focus on functionals $f: X \rightarrow \mathbb{R}$, where $X$ is assumed to have a triangulation $T$. If $f:X \rightarrow \mathbb{R}$ is such a  functional on $X$, we say $f$ is p.l. on $T$ if $f \circ h$ is p.l. on $T$.

\begin{lemma}\label{lem:linear sum}
    Let $T$ be a triangulation of $X$. Let $\{f_\lambda: X \rightarrow \mathbb{R}\}_{\lambda \in \Lambda}$ be a set of p.l., functionals on $X$ such that for all $x \in X$, $f_\lambda(x) =0$ for all but finitely many $\lambda \in \Lambda$. Then $f:= \sum\limits_{\lambda \in \Lambda} f_\lambda$ is p.l. on $T$. 
\end{lemma}

\begin{proof}
    Let $x \in X$, and let $\sigma= \langle v_0, ...v_p\rangle \in T$ be the  simplex in $T$ of minimal degree such that $x \in \sigma$. Let the Barycentric coordinates of $x$ be given by $x=\sum\limits_{i=0}^p a_i \cdot v_i$. Since $\sigma$ is minimal, this implies that $a_i >0$ for all $i$.
    
    By assumption, there exists a finite subset of $\Lambda$ consisting of all $f_\lambda$ that are nonzero on $x$. Furthermore, there exists a finite set of $f_\lambda$ that are nonzero on each of the vertices of $\sigma$. Let $\bar{\Lambda}$ be the finite set of all $\lambda$ such that  $f_\lambda$ is nonzero on $x$ or any of the $p$ vertices of $\sigma$. \\

    For $\lambda \notin \bar{\Lambda}$, $f_\lambda(x) =0=f_\lambda(v_i) \ \forall i$. Thus, $f(x) = \sum\limits_{\lambda \in \bar{\Lambda}} f_\lambda(x)$ and $f(v_i) = \sum\limits_{\lambda \in \bar{\Lambda}} f_\lambda(v_i)$. Therefore,  

    \begin{equation*}
        \begin{split}
            f(x)&= \sum\limits_{\lambda \in \bar{\Lambda}} f_\lambda(x)\\
            &= \sum\limits_{\lambda \in \bar{\Lambda}} \left(\sum\limits_{i=0}^p a_i \cdot f_\lambda(v_i)\right)\\
            &= \sum\limits_{i=0}^p a_i \left(\sum\limits_{\lambda \in \bar{\Lambda}} f_\lambda(v_i)\right)\\
            &= \sum\limits_{i=0}^p a_i \cdot f(v_i)
        \end{split}
    \end{equation*}

    Hence $f$ is p.l. on $T$. 
\end{proof}

We now move on from the juvenile notion of a piece-wise linear triangulation, to the adolescent idea of a nested sequence of p.l. triangulations. We say the pair of simplicial complexes $(T,S)$ is a p.l. triangulation of pair $(X,A)$ if $S \subset T$, $T$ is a triangulation of $X$ and $S$ is a triangulation of $A$.

\begin{definition}\label{def:nested triangulation}
    Let $(X,A)$ be a polyhedral pair in  $\mathbb{R}^d$. We say $\{(T^n, S^n)\}_{n=0}^\infty $ is a nested triangulation of $(X,A)$ if, \\

    \begin{itemize}
        \item $\forall n$, $(T^n, S^n)$ is a triangulation of $(X,A)$\\
        \item $\forall n$, $T^{n+1}$ is a refinement of $T^n$\\
        \item $\sup\limits_{\sigma \in T^n} diam(\sigma) \rightarrow 0$ as $n \rightarrow \infty$\\
    \end{itemize}
\end{definition}

\begin{example}\label{exampel: nested triangulation}
    We again return to the CFK triangulation of Example \ref{example:CFK triangulation}, where the polyhedral pair $(X,A) = (\mathbb{R}^2_\leq , \Delta)$. The CFK triangulation may be scaled by a factor of $\frac{1}{z}$ for integer $z>1$, to yield a triangulation formed by lines of the form $x,y = \frac{1}{z}\mathbb{Z}$ and $y=x+\frac{1}{z} \mathbb{Z}$. This forms the CFK triangulation at scale $\frac{1}{z}$, which is a refinement of the CFK triangulation at scale $1$. 

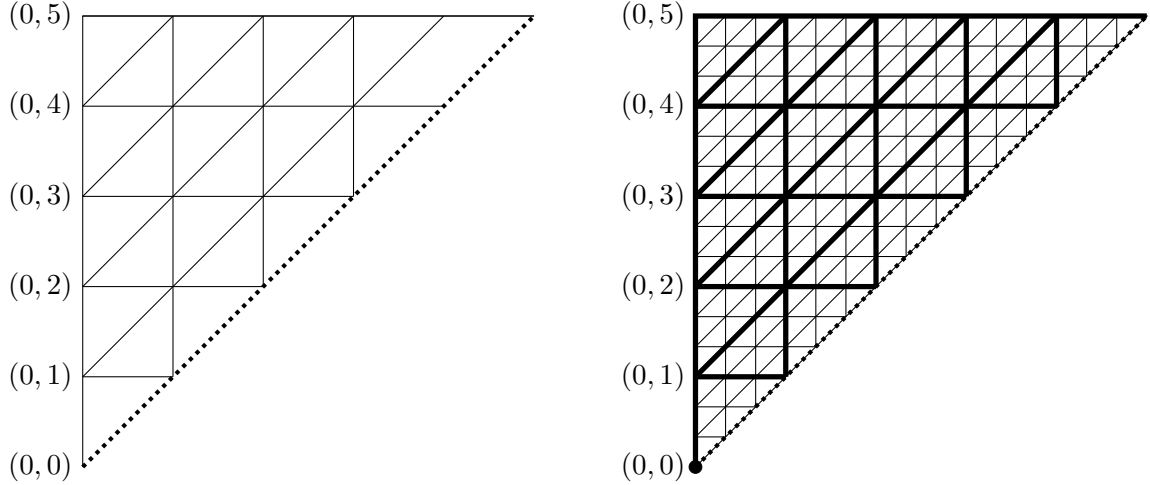
\begin{figure}[H]
    \begin{tikzpicture}[scale=1.2]
            \draw[dotted, line width=0.55mm] (0,0)--(5,5);
            \draw (5,5)--(0,5)--(0,0);
            \draw (0,2)--(2,2);
            \draw (0,1)--(1,1);
            \draw (0,3)--(3,3);
            \draw (0,4)--(4,4);
            \draw (0,5) --(5,5);
            \draw (1,1)--(1,5);
            \draw (2,2)--(2,5);
            \draw (3,3)--(3,5);
            \draw (0,1) -- (4,5);
            \draw (0,2) -- (3,5);
            \draw (0,3) -- (2,5);
            \draw (0,4) -- (1,5);
            \node [black, left] at (0,0){$(0,0)$};
            \node [black, left] at (0,1){$(0,1)$};
            \node [black, left] at (0,2){$(0,2)$};
            \node [black, left] at (0,3){$(0,3)$};
            \node [black, left] at (0,4){$(0,4)$};
            \node [black, left] at (0,5){$(0,5)$};
\end{tikzpicture}\hspace{1cm}\begin{tikzpicture}[scale=1.2]
         \draw[dotted, line width=0.55mm] (0,0)--(5,5);
            \draw[line width=0.70mm] (5,5)--(0,5)--(0,0);
            
            \draw[line width=0.70mm] (0,2)--(2,2);
            \draw[line width=0.70mm] (0,1)--(1,1);
            \draw[line width=0.70mm] (0,3)--(3,3);
            \draw[line width=0.70mm] (0,4)--(4,4);
            \draw[line width=0.70mm] (1,1)--(1,5);
            \draw[line width=0.70mm] (2,2)--(2,5);
            \draw[line width=0.70mm] (3,3)--(3,5);
            \draw[line width=0.70mm] (4,4) -- (4,5);
            \draw[line width=0.70mm] (0,1) -- (4,5);
            \draw[line width=0.70mm] (0,2) -- (3,5);
            \draw[line width=0.70mm] (0,3) -- (2,5);
            \draw[line width=0.70mm] (0,4) -- (1,5);
            \filldraw [black] (0,0) circle (2pt);
            \node [black, left] at (0,0){$(0,0)$};
            \node [black, left] at (0,1){$(0,1)$};
            \node [black, left] at (0,2){$(0,2)$};
            \node [black, left] at (0,3){$(0,3)$};
            \node [black, left] at (0,4){$(0,4)$};
            \node [black, left] at (0,5){$(0,5)$};

            \foreach \i in {0,1,...,15} {
        \draw (0,\i/3)--(5-\i/3,5);
    }

            \foreach \i in {0,1,...,15} {
            \draw (\i/3, \i/3) --(\i/3,5);
            \draw (0,\i/3)--(\i/3,\i/3);
            }
    \end{tikzpicture}

    \caption{ A Refinement of the Coxeter Freudenthal Kuhn triangulation of $(\mathbb{R}^2_\leq , \Delta)$ at scale $\frac{1}{3}$}
    \label{example: nested triangulation}
\end{figure}

We may continue to take refinements of the CFK triangulation, each at scale $\frac{1}{z}$ to the previous refinement. That is, define $T^n$ to be the CFK triangulation of $\mathbb{R}^2_\leq $ at scale $\frac{1}{z^n}$. 

For each $n$, let $S^n$ be the subcomplex of $T^n$ consisting of simplices contained in $\Delta$. Then for each $n$, $(T^n,S^n)$ is a triangulation of $\mathbb{R}^2_\leq$, and $T^{n+1}$ is a refinement of $T^n$. Lastly, for each $n$, $M_n=\sup\limits_{\sigma \in T^n} diam(\sigma) = \frac{\sqrt{2}}{z^{n}}$. Thus $M_n$ decreases to $0$ as $n \rightarrow \infty$, and hence $\{(T^n,S^n)\}_{n=0}^\infty$ is a nested triangulation of $(\mathbb{R}^2_\leq, \Delta)$.

\end{example}

We now consider functionals on a polyhedral pair $(X,A)$ endowed with a nested triangulation $\{(T^n, S^n)\}_{n=0}^\infty$. When the nested triangulation is clear, we will say $f:(X,A) \rightarrow (\mathbb{R},0)$ is $n-$linear, to mean that $f$ is p.l. linear on $T^n$. 

\begin{lemma}\label{lem:n2n+1linear}
    Let $\{(T^n,S^n)\}_{n=0}^\infty $ be a nested triangulation of polyhedral pair $(X,A)$. Let $f:(X,A) \rightarrow (\mathbb{R},0)$ be a functional on $X$. If $f$ is $n$-linear, then $f$ is $n+1$-linear, and hence $f$ is $m-$linear for all $m \geq n$. 
\end{lemma}

\begin{proof}
     Let $x \in X$, and $\sigma = \langle v_0, ...v_p\rangle$ be the simplex of minimal degree in $T^n$ such that $x \in \sigma$.  Suppose $x$ has Barycentric coordinates 
$x = \sum\limits_{i=0}^p a_i \cdot v_i$
    
   Similarly, let $\tau= \langle w_0, ...w_q\rangle $ be the minimal simplex in $T^{n+1}$ such that $x \in \tau$. Since $T^{n+1}$ is a refinement of $T^n$, this implies that $\tau\subset \sigma \subset X$. 

\begin{figure}[H]
    \centering
   \begin{tikzpicture}
       \draw (0,0)--(4,-1)--(5,3)--(0,0);
       \draw (3,-0.75)--(4.5,1)--(1.66,1)--(3,-0.75);
       \node[black,left] at (0,0){$v_1$};
       \node[black,right] at (4,-1){$v_2$};
       \node[black,above] at (5,3){$v_0$};
       \node[black,below] at (3,-0.75){$w_1$};
       \node[black,right] at (4.5,1){$w_2$};
       \node[black,above] at (1.66,1){$w_0$};
       \filldraw[black] (3,0) circle (2pt);  
       \node[black,above] at (3,0){$x$};
   \end{tikzpicture}
    \caption{Example using $p=2$}
    \label{fig:enter-label}
\end{figure}
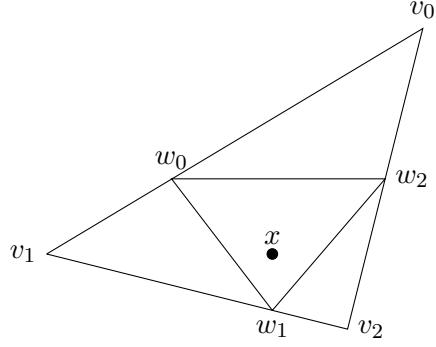

    Since $\tau$ is contained in $\sigma$, each $w_j$ has Barycentric coordinates in $\sigma$ such that $w_j= \sum\limits_{i=0}^p c_{i,j} \cdot v_i$. \\

    Suppose $x$ has Barycentric coordinates in $\tau$ such that $x= \sum\limits_{j=0}^q b_i w_i$. \\

    Now observe that

    \begin{align*}
        x &= \sum\limits_{i=0}^p a_i v_i \\
        &= \sum\limits_{j=0}^q b_j w_j\\
        &= \sum\limits_{j=0}^q b_j \left(\sum\limits_{i=0}^p c_{i,j} v_i\right)\\
        &= \sum\limits_{i=0}^p \left(\sum\limits_{j=0}^q b_j c_{i,j} \right) v_i
    \end{align*}

    By uniqueness of Barycentric coordinates of $x$ in $\sigma$, this implies that $a_i = \sum\limits_{j=0}^q b_j c_{i,j} \ \forall i$. Therefore, since $f$ is $n$-linear, we know $f(w_i) = \sum\limits_{j=0}^p c_{i,j} f(v_j)$. Furthermore,

    \begin{align*}
        f(x) &= \sum\limits_{i=0}^p a_i f(v_i) \\
        &= \sum\limits_{i=0}^p \left( \sum\limits_{j=0}^q b_jc_{i,j} \right)f(v_i)\\
        &= \sum\limits_{j=0}^q b_j \left(\sum\limits_{i=0}^p c_{i,j} f(v_i)\right)\\
        &= \sum\limits_{j=0}^q b_j f(w_j)
    \end{align*}

    Thus $f$ is $n+1$-linear. 
\end{proof}

\section{Constructing a Schauder basis}\label{sec:Schauder Basis}
We can now construct a Schauder Basis of the normed vector space $(Lip_c(X,A), ||\cdot ||_\infty)$ of compactly supported Lipschitz functionals on $X$ which are $0$ on $A$. We begin with our first method of constructing such a Basis, using piece-wise linear functionals respective to nested triangulations. 

\subsection{A Schauder Basis of Piece-wise Linear Functionals}

\begin{definition}\label{def:kernel}
    Let $\{(T^n,S^n)\}_{n=0}^\infty$ be a nested triangulation of polyhedral pair $(X,A)$. Furthermore, let $(L_n)_{n=0}^\infty$ be a sequence of positive real numbers such that $\sum\limits_{n=0}^\infty L_n =L < \infty$.  Let $V^n$ denote the collection of all vertices of $T^n$ that are not vertices of $S^n$. Let $V^{(n)}= V^n \backslash V^{n-1}$ be the set of such vertices that appear at layer $n$ of the nested triangulation.  For each $n$ and $v \in V^{(n)}$, we define the functional $\K_v: (X,A)\rightarrow (\mathbb{R},0)$ to be the unique function with the following properties. \\

    \begin{itemize}
        \item  $\K_v$ is $n-$linear. \\
        \item  $\K_v(v')=0$ for all $v'\in V^n$ s.t.  $v'\neq v$\\
        \item  $Lip(\K_v)= L_n$
    \end{itemize}
\end{definition}

\vspace{0.5cm}

    We will often refer to $V= \bigcup\limits_{n=0}^\infty V^{(n)}$ as the set of all vertices of the triangulations. We illustrate below that a unique function exists with the above properties for each vertex $v \in V$.\\

    \begin{proposition}\label{prop:kernels exist}

    Let $T$ be a triangulation of a polyhedron $X \subset \mathbb{R}^d$ and $L>0$. For $v$ a vertex of $T$, there exists a unique p.l. function $f: X \rightarrow \mathbb{R}$ such that $f(v') =0$ for all vertices $v' \neq v$  and $Lip(f) = L$.  \\

\begin{proof}

    For $f$ to be p.l. on $T$, it must be determined by it's values on vertices. Hence if $f$ is $0$ on all vertices not equal to $v$, then $f$ remains to be determined solely by its value at $v$. Thus it is sufficient to show that such a choice of $f(v)$ can be made such that $Lip(f)=L$. 
    
    For $c >0$,  let $f_c: X \rightarrow \mathbb{R}$ be the p.l. functional, which is  $0$ on all vertices not equal to $v$, and $f_c(v)=c$. Since $f_c$ is 0 on all simplices not containing $v$ as a vertex, $Lip(f)$ is determined by its restriction to simplices containing $v$.

    Let $\sigma \in T$ be such that $v$ is a vertex of $\sigma$. Without loss of generality, assume $\sigma = \langle v_0, v_1...v_d\rangle$ such that $v_0=v$. Then consider the restriction $g_c= f_c|_{\sigma}$. Let $x \in \sigma$, with barycentric coordinates $x= \sum\limits_{i=0}^p a_i v_i$. Then $f_c(x)=g_c(x) = \sum\limits_{i=0}^p a_i f_c(v_i) = a_0 \cdot c$. Thus the level sets of $g_c$ are contained in hyperplanes of the form $H_t= \{x= \sum\limits_{i=0}^d a_i \cdot v_i \in \sigma \ | \ \sum\limits_{i=1}^p a_i = 1-t\}$ where $t$ ranges from $0$ to $1$. Hence the gradient of $g_c$ is a normal vector to any of these level set hyperplanes.

    \begin{figure}[H]
        \centering
        \includegraphics[width=0.5\linewidth]{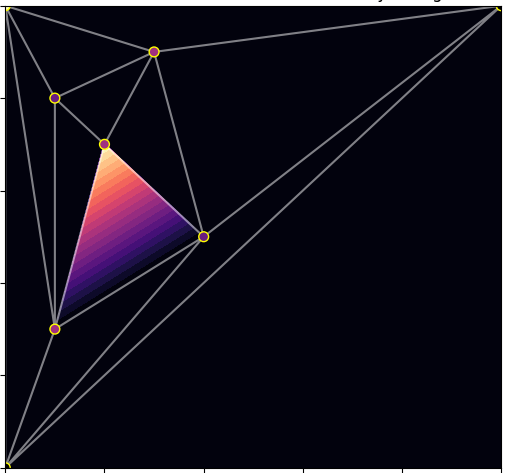}
        \caption{Example graph of $g_c$}
        \label{fig: restriction of kernel}
    \end{figure}

    Let $\bar{g}_c: \mathbb{R}^d\rightarrow \mathbb{R}$ be the linear extension of $g_c$ to all of $\mathbb{R}^d$. That is, using linearly independent $\{v_i-v_0\}_{i=1}^d$ as a basis of $\mathbb{R}^d$, $$\bar{g}_c: v_0 + \sum\limits_{i=1}^d a_i \cdot (v_i-v_0) \mapsto \left(1-\sum\limits_{i=1}^d a_i\right) \cdot c$$



    
    Then $Lip(\bar{g}_c) = Lip(g_c)$. Let $y$ be the projection of $v_0$ onto hyperplane $H_0= \bar{g}_c^{-1}(0)$.   Thus the change in $f_c|_\sigma$ is maximal along the vector $\langle v_0-y\rangle$. Then $Lip(f|_\sigma) = Lip(\bar{g}_c) = \frac{f(v)-f(y)}{d(v,y)}= \frac{c}{d(v,H_0)}$. 

    We now repeat this process for each simplex $\sigma \in T$ containing $v$ as a vertex. For each such $\sigma \in T$,  let $y_\sigma$ be the projection of $v$ onto the hyperplane containing the face of $\sigma $ opposite $v$. Then let $c^* = \min\limits_{\sigma \ | \ v \in \sigma } d(v,y_\sigma) \cdot L$. Note that by Remark \ref{Remark: finite simplices}, this is indeed a minimum not an infemum. Thus $c^*>0$, and this makes $Lip(f_{c^*}) =L$. 

\end{proof}

 \end{proposition}

 \begin{example}\label{example:CFK kernels}
Below we illustrate the collection of functionals $\{\K_v\}_{v \in V}$ formed by a CFK-nested triangulation of $(\mathbb{R}^2_\leq , \Delta)$, where the $n$th triangulation $(T^n, S^n)$ is the CFK-triangulation at scale $\frac{1}{2^n}$. In this case, the vertices of $V^{(n)}$ are points of the form $\left(\frac{a}{2^n}, \frac{b}{2^n}\right)$ where $a < b \in \mathbb{Z}$ and either $a$ or $b$ is odd. The distance of a vertex $v=\left(\frac{a}{2^n}, \frac{b}{2^n}\right)$ to the closest hyperplane containing it's opposite face in one of the simplices containing $v$ is $\frac{1}{2^n\sqrt{2}}$. Thus, for a functional $\K_{\left(\frac{a}{2^n}, \frac{b}{2^n}\right)}$ defined to have Lipschitz constant $\frac{1}{2^n}$, we may determine that $\K_{\left(\frac{a}{2^n}, \frac{b}{2^n}\right)}\left(\frac{a}{2^n}, \frac{b}{2^n}\right) = \frac{1}{\sqrt{2}\cdot2^{2n}}$. 
    \begin{figure}[H]
        \centering
        \includegraphics[width=0.5\linewidth]{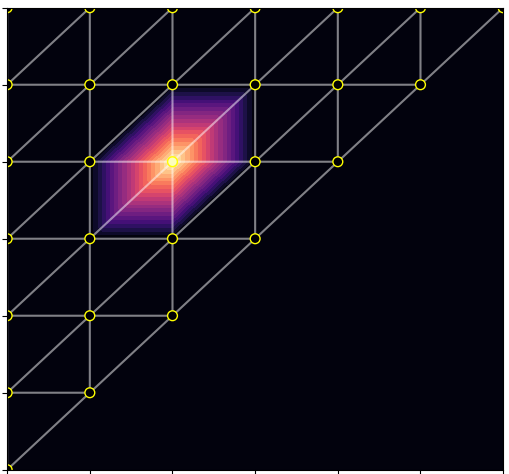}\includegraphics[width=0.5\linewidth]{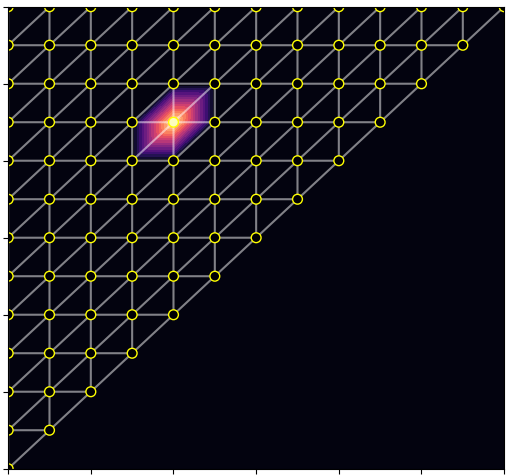}
        \caption{\textbf{(Left)} Functional $\K_{(2,4)}$ of layer 0 , \textbf{(Right)} Functional $\K_{(2, 4.5)} $ of layer 1}
        \label{fig:CFK kernels}
    \end{figure}


\end{example}


\begin{theorem}[Schauder Basis]\label{thrm: Schauder Basis}
    Let $\{(T^n, S^n)\}_{n =0}^\infty$ be a nested triangulation of polyhedral pair $(X,A)$ in $\mathbb{R}^d$. Let $V= \bigcup\limits_{n=0}^\infty V^{(n)}$ be the set of all vertices of these triangulations that are not in $A$.  Let $\B= \{\K_v \ | \ v \in V\}$ be the set of all functionals on $(X,A)$ defined as in Definition \ref{def:kernel}, with Lipschitz constants $(L_n)_{n=0}^\infty , \ \sum\limits_{n=0}^\infty L_n = L$. Then there exists an ordering of $\B$ such that $\B$ is a Schauder basis of $Lip_c(X,A)$ under the $\ell^\infty$ norm.
\end{theorem}

\begin{proof}

     First we consider the case that $X$ is compact. In this case, for all $n\geq 0$, $|V^{(n)}|< \infty$ by Remark \ref{Remark: finite simplices}. Thus we order $V$ (and equivalently $\B$) lexicographically by $(n,x_1, x_2, ...x_d)$ where $v=(x_i)_{i=1}^d$ and $n$ is such that $v \in V^{(n)}$.  Let $M_n= \sup\limits_{\sigma \in T^n} diam(\sigma)$ for all $n$. By Definition \ref{def:nested triangulation}, $M_n$ decreases to $0$ as $n\rightarrow \infty$. 
    
    Let $f \in Lip_c(X,A)$. We need define $\{a_v\}_{v \in V}$ such that $\sum\limits_{v \in V} a_v \cdot \K_v = \sum\limits_{n=0}^\infty \sum\limits_{v \in V^{(n)}} a_v \cdot \K_v$ converges uniformly to $f$.\\

    We begin with vertices in $V^{(0)}$. Let $v \in V^{(0)}$, then define $a^0_v := \frac{f(v)}{\K^0_v(v)}$. Define $f^0:=\sum\limits_{v \in V} a_v \cdot \K_v$. \\

    Recall $\K^0_v(v') =0$ for all $v' \in V^0 \ $ such that $\ v' \neq v$. This implies that $f^0|_{V^0}= f|_{V^0}$. Furthermore, $f^0$ is p.l. on $T^0$ by Lemma \ref{lem:linear sum}. 

    For $n \geq 1$, and $v \in V^{(n)}$, we define $a_v := \frac{f(v) - f^{n-1}(v)}{\K_v(v)}$. Then define $f^n := f^{n-1} + \sum\limits_{v \in V^{(n)}} a_v \K_v$.  Again, each $\K_v(v') =0 \ \forall v' \in V^n \ $ such that $v' \neq v$. Thus, $f^n|_{V^n} = f|_{V^n}$. \\

    Furthermore, for all $n$, $f^n$ is $n$-linear by Lemmas \ref{lem:linear sum} and \ref{lem:n2n+1linear}.\\

    Now we show that $f^n $ converges to $f $ uniformly as $n \rightarrow \infty$. Fix $n \geq 0$ and let $x \in X$. Let $\sigma = \langle v_0, ...v_d\rangle \in T^n$ such that $x \in \sigma$ with Barycentric coordinates $x = \sum\limits_{i=0}^d c_i \cdot v_i$. Then since $f^n$ is p.l. on $T^n$, and equals $f$ on $V^n$, we have that 
    
    $$f^n(x) = \sum\limits_{i=0}^p c_i f^n(v_i) = \sum\limits_{i=0}^p c_i f(v_i)$$
    
  Then we have the following, \\

    \begin{equation*}
        \begin{split}
            |f(x) - f^n(x) | &= |f(x) - \sum\limits_{i=0}^p c_i f(v_i)|\\
            &= |\sum\limits_{i=0}^p c_i (f(x) - f(v_i))|\\
            &\leq \sum\limits_{i=0}^p c_i |f(x) - f(v_i)|\\
            &\leq \sum\limits_{i=0}^p c_i \cdot Lip(f) \cdot d(x,v_i)\\
            &\leq \sum\limits_{i}^p c_i \cdot Lip(f) \cdot M_n\\
            &= Lip(f) \cdot M_n\\
        \end{split}
    \end{equation*}

    By Definition \ref{def:nested triangulation}, $M_n$ decrease to $0$ as $n $ goes to $\infty$. Therefore, $f^n$ converges uniformly to $f$. \\

    While we have shown the existence of scalars $\{a_v\}_{v \in V}$ such that $\sum\limits_{v \in V} a_v \K_v$ converges uniformly to $f$, it remains to show that the choices of $a_v$ are unique.

    Suppose $f = \sum\limits_{n =0}^\infty \sum\limits_{v \in V^{(n)}} b_v \K_v$. We show that $b_v = a_v$ as defined above for all $n $ and $v \in V^{(n)}$. \\

    Begin with vertices in $V^0$. For $v \in V^0$, note again that for all $v' \in V$, $\K_{v'}(v) =0$  for all $v' \in V^{(0)} $ with $v' \neq v$. Furthermore, $v \notin V^{(n)} \ \forall n >0$. Hence, $\forall v' \in V$,  $\K_{v'}(v)=0$ unless  $v'=v$.   Therefore, if $b_v\neq a_v = \frac{f(v)}{\K_v(v)}$, then  $f^0(v) \neq f(v)$, and $f^n(v) = f^0(v) \neq f(v)$ for all $n \geq 1$. Thus $f^n$ will not converge to $f$. 
    
    Hence conclude $b_v = a_v $ for all $v \in V^0$. We continue by induction. Fix $n$ and suppose that for $m <n$, and $v \in V^{(m)}$, $a_v=b_v$. Let $f^{n-1}$ be defined as before. Then $f(v) = f^{n-1}(v)$ for all $v \in V^{n-1}$. Consider $g:= f-f^{n-1}$. 

    Fix $v \in V^{(n)}$. Then $v\in V^m $ for all $m \geq n$. For such $m\geq n$, recall for any $v' \in V^{(m)}$, $\K_{v'}(v) =0$, unless $m=n$ and $v' =v$. Hence $\K_v$ is the only functional of $\{\K_v \ | \ v \in V^{(m)} \}_{m \geq n}$ which is nonzero on $v$. That is, since all coefficients of functionals $\{\K_{v'}\ | \ v' \in V^{(m)} \ | \ m <n\}$ is already fixed at $b_{v'} = a_{v'}$, then the only remaining coefficient that can alter the sum of functionals on $v$ is $b_v$. Therefore, $b_v = \frac{f(v) - f^{n-1}(v)}{\K_v(v)} = a_v$. Hence the sequence of scalars $\{a_v\}_{v \in V}$ is unique for every $f \in Lip_c(X,A)$, and thus $\B$ is a Schauder Basis of $Lip_c(X,A)$.

    Now we consider the case that $X$ is not compact. In this case, we must be a little more creative in our ordering of $\B$. We will again use the coordinates of the vertices of $V$, along with the unique $n$ for which these vertices exist in $V^{(n)}$, both of which were used in defining elements of $\B$. 

    For each integer $N \geq 1$, define the Rafter at $N$, $Raft(N)$, by $Raft(1) = \{x \in X  \ | \ ||x||_\infty \leq 1\}$ and  $Raft(N)= \{x \in X  \backslash Raft(N-1)  \ s.t. \ ||x||_\infty \leq N  \}$ for $N > 1$. We order $\B$ with these rafters as follows. 

    Begin by ordering elements $\K_v \in \B$ centered at vertices $v \in V^{0}$, with $v \in Raft(1)$. Call this set $\B^{(0,1)}$. Order $\B^{(0,1)}$ lexicographically by coordinates of $v$. Note since $Raft(1)$ is compact, then $\B^{(0,1)}$ is finite. 

    Next, let $\B^{(0,2)}$ be the set of elements $\K_v \in \B$ such that $v \in V^{(0)} \cap (Raft(2) )$. Let $\B^{(1,1)} $ be the subset of $\B$ consisting of functionals $\K_v$ such that $v \in V^{(1)} \cap Raft(1)$. 

    Order both the finite sets $\B^{(0,2)}$ and $\B^{(1,1)}$ lexicographically by coordinates in $\mathbb{R}^d$. 

    We may continue to define sets $\B^{(M,N)}= \{\K_v \in \B \ | \ v \in V^{(M)} \cap Raft(N) \}$ for all $M$ and $N$. This partitions $\B$ into totally ordered finite subsets. 

    Now we define a total ordering of the collection of finite sets $\{\B^{(M,N)}\}_{M,N}$, which induces a total ordering of $\B$. If $N>1$, then $\B^{(M,N)}$ is followed by $\B^{(M+1, N-1)}$. If $N=1$, then $\B^{(M,N)}$ is followed by $\B^{(0, N+1)}$. We now prove that $\B$ forms a Schauder Basis with this ordering.

    Let $f \in Lip_c(X,A)$. Denote $\B= \{\K_{v_i}\}_{i=0}^\infty$ be the total ordering of $\B$ as defined above. \\

    Let $x \in X \backslash A$. Fix some $n>0$ such that the minimal simplex of $\sigma_n$ of $T^n$ containing $x$ is contained in a single rafter. Choose $N$ large enough such that the finitely many vertices of $V^{n}$ contained in $supp(f)$ have index $\leq N$ in the ordering of $\B$ above. \\
    
    Let $a_{v_i}$ be defined as before, in the case that $X$ was compact. Then define $g^N = \sum\limits_{i=0}^N a_i \cdot \K_{v_i}$. 
    Let $f^n$ be defined as in the previous case, where $f^n$ is the n-linear approximation of $f$ by functionals of layer $\leq n$. Note this a finite sum of nonzero functionals since $f$ has compact support. 

    By the previous case in which $X$ was compact, we know that $|f^n(x) -f(x) | \leq M_n \cdot Lip(f)$. 

    Furthermore, $g^N- f^n$ is $0$ on $V^n$. Let $m$ be maximal such that there exists a vertex of $V^{(m)}$ of index $\leq N$. Then $g^N$ is linear on $T^m$.
    
    Next we may write $g^N(x)= f^n(x) + \sum\limits_{p=n+1}^m \sum\limits_{v \in \sigma_p} \bar{c}_{v} (f(v)-f^{p-1}(v))$ where $\sigma_p$ is minimal simplex of $T^p$ containing $x$, and $\bar{c}_v$ is the Barycentric coordinate of $x$ relative to $v$ in $\sigma_p$ if $v $ has index $\leq N$ in $\B$,  and $\bar{c}_v =0$ otherwise. 

    We know from the previous case in which $X$ is compact, that $|f(v)-f^{p-1}(v)| \leq M_{p-1}\cdot Lip(f)$ for all $v \in V$. Therefore, 

    \begin{equation*}
        \begin{split}
            |g^N(x) - f(x) | &\leq |g^N(x) - f^n(x)| +  \sum\limits_{p=n+1}^m \sum\limits_{v \in \sigma_p} \bar{a}_{v} |(f(v)-f^{p-1}(v))|\\
            &\leq M_n \cdot Lip(f) + \sum\limits_{p=n+1}^m M_{p-1}\cdot Lip(f)\\
            &\leq \left( M_n + \sum\limits_{p=n}^\infty M_p \right) \cdot Lip(f)\\
        \end{split}
    \end{equation*}

    Since $\sum\limits_{n=0}^\infty M_n < \infty$, the above converges to $0$ as $n \rightarrow \infty$, as so also as $N \rightarrow \infty$.  

\end{proof}



\begin{definition}
    Let $\B$ be a Schauder basis of $Lip_c(X,A)$ with the $\ell_\infty$ metric. We say $\B$ meets the \emph{locally Lipschitz finite} property iff there exists some $M< \infty$ such that for every $x \in X$, $\sum\limits_{f \in \B \ | \ f(x) \neq 0} Lip(f) \leq M$. 
\end{definition}
\begin{lemma}\label{lemma: LLF}
    Let $\mathbb{B}$ be the Schauder basis associated to a nested triangulation $\{(T^n,S^n)\}_{n=0}^\infty$ of $(X,A)$ where $X \subset \mathbb{R}^d$. Denote the  Lipschitz constants of the functionals in $\B$ by $\{L_n\}_{n=0}^\infty$, $\sum\limits_{n=0}^\infty L_n =L< \infty$. Then $\B$ meets the locally Lipschitz finite property with upper bound $M= L \cdot (d+1)$. 
\end{lemma}

\begin{proof}
    Let $x \in X$ and $n \geq 0$. Then  $x$ is in the interior of exactly 1 simplex in $T^n$. Let $\sigma_n = \langle v_0^n, ...v^n_{p_n}\rangle $ be this simplex, where $p_n \leq d$. Then there exists Barycentric coordinates $\{a_i\}_{i=0}^{p_n}$ such that $x = \sum_{i=0}^{p_n} a_i v^n_i$. \\

    Let $\tau= \langle w_0, ...w_q\rangle$ be any simplex of $T^n$ such that $x \in \tau$. Then since $x$ is in the interior of $\sigma_n$, $\sigma_n$ must be a face of $\tau$. Furthermore, if $\{b_j\}_{j=0}^q$ are the Barycentric coordinates of $x$ in $\tau$, then $b_j=0$ iff the associated vertex $w_j$ of $\tau$ is not a vertex of $\sigma_n$. Furthermore, since every functional at layer $n$ is piece-wise linear on $T^n$, then the only functionals which are nonzero on $x$ are precisely the ones which are nonzero on at least one vertex of $\sigma_n$. But each functional at layer $n$ is nonzero at precisely 1 vertex of $T^n$, thus the functionals which are nonzero on $x$ are precisely $\{\K_{v^n_i}\}_{i=0}^{p_n}$ for vertices $\{v^n_i\}_{i=0}^{p_n} $ of $\sigma_n$.

    Therefore, $\sum\limits_{\K_v \in \B \ | \ \K_v(x) \neq 0} \K_v = \sum\limits_{n=0}^\infty \sum\limits_{i=0}^{p_n} \K_{v^n_i}$, where  $\forall n$,  $\sigma_n$ is the unique simplex of $T^n$ such that $x \in Int(\sigma)$. \\

    Since $\sigma$ is a simplex of a triangulation of $\mathbb{R}^d$, then $p_n\leq d$ for all $n$, and hence there are at most $d+1$ functionals at layer $n$ which are nonzero on $x$. 

    Lastly, recall that at layer $n$, a functional has Lipschitz constant $L_n$ as from the definition \ref{def:kernel}, and $\sum\limits_{n=0}^\infty L_n =L< \infty$. Then we have that, 

    $$\sum\limits_{\K_v \in \B \ | \ \K_v(x) \neq 0} Lip(\K_v) = \sum\limits_{n=0}^\infty \left(\sum\limits_{i=0}^{p_n} Lip(\K_{v_i^n})  \right)\leq \sum\limits_{n=0}^\infty (d+1) \cdot L_n = (d+1)\cdot L$$

    Thus $\B$ meets the locally Lipschitz finite property with upper bound $M=(d+1)\cdot L$. 
\end{proof}

\subsection{Stacked Functionals on Nested Coxeter-Freudenthal-Kuhn Triangulations}

The previously described method is not the only way we can use nested triangulations to define a Schauder basis of $Lip_c(X,A)$. We propose one more method, having an additional property outlined in sections 5 and 6, namely that the induced vectorization takes as input an unsigned persistence diagram $\alpha$, and maps it to a vector whose norm equals $W_1(\alpha, \emptyset)$. \\

\begin{definition}\label{def:stacked functional}
    Let  $X$ be a compact polyhedron in $\mathbb{R}^d$ and $\{(T^n, S^n)\}_{n=0}^\infty$ a nested CFK-triangulation, where $T^n$ is the CFK-triangulation at scale $\frac{1}{z^n}$ for some $z \geq 2$. For all $n \geq 0$ and $v \in V$, let $\K_v^n$ denote the n-linear functional, nonzero at v, as defined in Definition \ref{def:kernel}, having Lipschitz constant $\frac{1}{z^n}$. Note we are including vertices $v \notin V^{(n)}$ in this context.\\

    For any $v \in V$, let $N$ be minimal such that $v \in V^{(N)}$. Define the \emph{Stacked Functional} to be $$\mathfrak{K}_v:= \frac{\sqrt{2}\cdot d(v,A) (z^2-1)}{z^2}\sum\limits_{n=N}^\infty \K_v^n$$

\end{definition}

We note a few useful properties of these functionals. Recall that for functionals $\K_v^n$, $Lip(\K_v^n) = \frac{1}{z^n}$. Therefore, $Lip(\mathfrak{K}_v) = \frac{\sqrt{2}\cdot d(v,A) (z^2-1)}{z^2}\sum\limits_{n=N}^\infty \frac{1}{z^n}= \frac{\sqrt{2}\cdot d(v,A) (z^2-1)}{z^2} \cdot \frac{z^{1-N}}{z-1}$.\\

Also, recall that $\K_v^n(v) = \frac{1}{\sqrt{2}z^n} \cdot \frac{1}{z^n}= \frac{1}{\sqrt{2}z^{2n}}$. Therefore, $\mathfrak{K}_v(v) = \frac{\sqrt{2}\cdot d(v,A) (z^2-1)}{z^2} \sum\limits_{n=N}^\infty \frac{1}{\sqrt{2}\cdot z^{2n}}= \frac{ d(v,A) (z^2-1)}{z^2} \left( \frac{z^{2-2N}}{z^2-1}\right)= \frac{d(v,A)}{z^{2N}}$.\\

Similar to the Schauder basis  $\B= \{\K_v\}_{v \in V}$, if $v \in V^{(N)}$, then $\mathfrak{K}_v$ is $0$ on all  $v' \in V^{(N)}$ such that $v'\neq v$, and for all vertices $v' \in V^{(M)}$ for $M >N$. 

\begin{figure}[H]
    \includegraphics[width=0.4\linewidth]{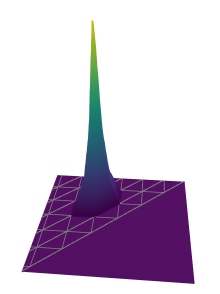}\includegraphics[width=0.3\linewidth]{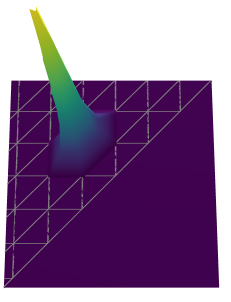}
    \caption{A plot of the stacked Kernel $\mathfrak{K}_{(2,4)}$ of the CFK-triangulation at scale 1}
    \label{fig:Seussian Kernel}
\end{figure}

\begin{theorem}[Stacked Schauder Basis] \label{thrm:stacked schauder basis} Suppose $X\subset \mathbb{R}^d$ is compact, and let $\{(T^n, S^n)\}_{n=0}^\infty$ be a nested CFK-triangulation of scales $\frac{1}{z^n}$. Let $\B= \{\mathfrak{K}_v\}_{v \in V}$, ordered lexicographically by $N$ and coordinates of $v \in V^{(N)}$. Then $\B$ is a l.l.f. Schauder basis of $Lip_c(X,A)$. 
    
\end{theorem}

\begin{proof}
    
    Fix $f \in Lip_c(X,A)$. We need define coefficients $\{a_v\}_{v \in V}$ such that $f= \sum\limits_{v \in V} \mathfrak{K}_v$. Begin, as before, with vertices in $T^0$. \\

    For such $v \in V^0$, let $a_v := \frac{f(v)}{\mathfrak{K}_v(v)} = \frac{f(v)}{d(v,A)}$. Let $f^0 := \sum\limits_{v \in V^0} a_v \mathfrak{K}_v$. By construction, $f|_{V^0} = f^0|_{V^0}$. However, $f^0$ is not linear on $T^0$. We will still continue to define $f^N$ iteratively so that for $v \in V^{(N)}$, $a_v := \frac{f(v)-f^N(v)}{\mathfrak{K}_v(v)} \mathfrak{K}_v$. In so doing, $f|_{V^N} = f^N|_{V^N}$ for all $N \geq 0$. We need show that $f^N$ converges to $f$ as $N \rightarrow \infty$. \\

    To do this, we must define another series of functions closely related to $\{f^N\}_{N=0}^\infty$. 

    Note that $f^N:=\sum\limits_{m=0}^N \sum\limits_{v \in V^{(m)}} a_v \mathfrak{K}_v= \sum\limits_{m=0}^N \sum\limits_{v \in V^{(m)}} a_v  \frac{\sqrt{2}\cdot d(v,A) (z^2-1)}{z^2}\sum\limits_{n=m}^\infty \K_v^n$.

    Define $\hat{f}^N$ to be the subsum of functionals $\K_v^n$ of $n \leq N$. That is,  $$\hat{f}^N := \sum\limits_{m=0}^N \sum\limits_{v \in V^{(m)}} a_v \frac{\sqrt{2}\cdot d(v,A) (z^2-1)}{z^2}\sum\limits_{n=m}^N \K_v^n$$

    We claim the following. 

    \begin{lemma}
        For all $N \geq 0$, $||f-\hat{f}^N||_\infty \leq \left( \frac{\sum\limits_{n=0}^N z^n}{z^{2N}}\right) Lip(f)  + \frac{1}{z^{2N+2}} \sup(|f|)$. Hence $\hat{f}^N \xrightarrow{N} f$ uniformly.
    \end{lemma}

        \begin{proof}
            We prove this inductively on $N$. Begin with $N=0$. Let $x \in X$, and let $\sigma = \langle v_0, ...v_p\rangle$ be a simplex in $T^0$ containing $x$, such that $x$ has Barycentric coordinates $\{c_i\}_{i=0}^p$.  

            \begin{equation*}
                \begin{split}
                    |f(x) - \hat{f}^0(x)|&= |f(x) - \sum\limits_{i=0}^p c_i \hat{f}^0(v_i)| \hspace{0.5cm}\textbf{ because $\hat{f}^0$ is $0$-linear}\\
                    &\leq \sum\limits_{i=0}^p c_i \left(|f(x)-f(v_i)| + |f(v_i) - \hat{f}^0(v_i)|\right)\\
                    &\leq Lip(f) + \max\limits_{i} |f(v_i) - \hat{f}^0(v_i)| \\
                \end{split}
            \end{equation*}

            We now put a bound on $|f(v_i)-\hat{f}^0(v_i)|$. Recall that 

            \begin{equation*}
                \begin{split}
                    \hat{f}^0(v_i)&= a_{v_i} \left( \frac{\sqrt{2} \cdot d(v,A) (z^2-1)}{z^2}\right) \K_{v_i}^0(v_i)\\
                    &= \left(\frac{f(v_i)}{d(v,A)}\right)\left(\frac{\sqrt{2} \cdot d(v,A) (z^2-1)}{z^2}\right) \left(\frac{1}{\sqrt{2}}\right)\\
                    &= \left(\frac{z^2-1}{z^2}\right) f(v_i)
                \end{split}
            \end{equation*}

            Therefore, $|f(v_i)-\hat{f}^0(v_i)| \leq \frac{1}{z^2}\sup(|f|)$. Hence $ |f(x) - \hat{f}^0(x)| \leq Lip(f) + \frac{1}{z^2} \sup(|f|)= \left(\frac{\sum\limits_{n=0}^0 z^n}{z^{2 \cdot 0}}\right)Lip(f) + \frac{1}{z^{2\cdot 0 +2}} \sup(|f|)$. \\

            Now let $N >0$. Suppose that $||f-\hat{f}^{N-1}||_\infty \leq \left( \frac{\sum\limits_{n=0}^{N-1} z^n}{z^{2N-2}}\right) Lip(f)  + \frac{1}{z^{2N}} \sup(|f|)$. As before, let $x \in X$ be contained in a simplex $\sigma \in T^N$ such that $\sigma = \langle v_0, ...v_p\rangle$.  \\

            \begin{equation*}
                \begin{split}
                    |f(x)-\hat{f}^N(x) | &\leq \sum\limits_{i=0}^p c_i \left(|f(x)-f(v_i)| + |f(v_i) - \hat{f}^N(v_i)|\right)\\
                    &\leq \frac{1}{z^N} Lip(f) + \max\limits_i |f(v_i) - \hat{f}^N(v_i)|\\
                \end{split}
            \end{equation*}

            We again need place an upper bound on $|f(v_i) - \hat{f}^N(v_i)|$. \\

            \begin{equation*}
                \begin{split}
                    \hat{f}^N(v_i) &= \hat{f}^{N-1}(v_i) + a_{v_i} \left(\frac{\sqrt{2} d(v,A)(z^2-1)}{z^2}\right)\K_{v_i}^N(v_i)\\
                    &= \hat{f}^{N-1}(v_i) + \left( \frac{(f(v_i)-\hat{f}^{N-1}(v_i))z^{2N}}{d(v,A)}\right) \left(\frac{\sqrt{2} d(v,A)(z^2-1)}{z^2}\right)\left(\frac{1}{\sqrt{2} \cdot z^{2N}}\right)\\
                    &=  \hat{f}^{N-1}(v_i) + (f(v_i) -  \hat{f}^{N-1}(v_i) ) \left(\frac{z^2-1}{z^2}\right)\\
                    &= \frac{1}{z^2}  \hat{f}^{N-1}(v_i)  + \left(\frac{z^2-1}{z^2}\right) f(v_i)\\
                \end{split}
            \end{equation*}

            Therefore, $|f(v_i) - \hat{f}^N(v_i)| = \frac{1}{z^2} | f(v_i) - \hat{f}^{N-1}(v_i)|$. Hence we have the following. \\

            \begin{equation*}
                \begin{split}
                     |f(x)-\hat{f}^N(x) |&\leq \frac{1}{z^N} Lip(f) + \frac{1}{z^2} | f(v_i) - \hat{f}^{N-1}(v_i)|\\
                     &\leq \frac{1}{z^N} Lip(f) + \frac{1}{z^2} \left(\left( \frac{\sum\limits_{n=0}^{N-1} z^n}{z^{2N-2}}\right) Lip(f)  + \frac{1}{z^{2N}} \sup(|f|)\right) \\
                     &= \left(\frac{1}{z^N} + \frac{\sum\limits_{n=0}^{N-1} z^n}{z^{2N}} \right) Lip(f) + \frac{1}{z^{2N+2} }\sup(|f|)\\
                     &= \left(\frac{\sum\limits_{n=0}^{N} z^n}{z^{2N}} \right) Lip(f) + \frac{1}{z^{2N+2} }\sup(|f|)\\
                \end{split}
            \end{equation*}

            Note that $\frac{\sum\limits_{n=0}^{N} z^n}{z^{2N}}  = \frac{1}{z-1} \left(\frac{1}{z^{N-1}} -\frac{1}{z^{2N}}\right)$. Hence the above expression converges to $0$ as $N$ goes to $\infty$. \\

            Thus we have shown that $||\hat{f}^N -f||_\infty \xrightarrow{N} 0$. \\
            
            \end{proof}
            
            In order to show that $||f^N-f||_\infty \xrightarrow{N} 0$, we will show that $||\hat{f}^N-f^N||_\infty \xrightarrow{N} 0$. \\

            To show this, fix $N$ and consider $ f^N-\hat{f}^N$. Note then that by defintion, $f^N-\hat{f}^N$ is the sum of functionals $\K_v^n$ for $n \geq N+1$, centered at vertices $v \in V^{(N)}$.

                    $$f^N-\hat{f}^N=\sum\limits_{v \in V^{(N)}} \left( a_v \cdot \frac{\sqrt{2} \cdot d(v,A) (z^2-1)}{z^2}\right)\sum\limits_{n=N+1}^\infty \K_v^n$$

                    Note that for $v,v' \in V^{(N)}$ with $v \neq v'$, and $m \geq N+1$, that the supports of $\K_{v}^m $ and $\K_{v'}^m$ are disjoint. Therefore, since each $\K_v^n$ achieves its maximum as $v$ for all $v $and $n$, this implies that 

                    $$||f^N-\hat{f}^N||_\infty = \sup\limits_{v \in V^{(N)}} |f^N(v) - \hat{f}^N(v)|$$

                    However, recall that $f^N(v) = f(v)  \ \forall v \in V^{(N)}$. Hence 

                    $$||f^N-\hat{f}^N||_\infty =\sup\limits_{v \in V^{(N)}} |f(v) - \hat{f}^N(v)| $$

                    We have shown that $\hat{f}^N \xrightarrow{N} f$ uniformly. Thus we have shown that $||f^N-\hat{f}^N||_\infty \xrightarrow{N} 0$. Combining these results, we have shown that $f^N \xrightarrow{N} f$ uniformly. \\

            We now need verify that the collection of coefficients $\{a_v\}_{v \in V}$ are unique for the summations $f^N$ to converge to converge to $f$. Recall that for $v \in V^{(N)}$, and $v' \in V^{(M)}$ with $v' \neq v$ and $M \geq N$, we have that $\mathfrak{K}_{v'}(v)=0$. Hence, for the same argument as in the proof of Theorem \ref{thrm: Schauder Basis}, the coefficients $\{a_v\}_{v \in V}$ are indeed unique. Hence we conclude that $\B= \{\mathfrak{K}_v\}_{v \in V}$ is a Schauder basis of $Lip_c(X,A)$. \\

            The last part of our theorem requires that we verify that $\B$ is locally Lipschitz finite. Let $x \in X$, and fix $N \geq 0$. Then there is a unique simplex of minimal degree $\sigma \in T^N$ such that $x \in \sigma$. As in the case of Theorem \ref{thrm: Schauder Basis}, the only functionals of $\mathfrak{K}_v$ centered at vertices $v \in V^{(N)}$, are those such that $v \in \sigma$. Recall that $Lip(\mathfrak{K}_v) = \frac{\sqrt{2}\cdot d(v,A) (z^2-1)}{z^2} \frac{z^{1-N}}{z-1}$. Since there are at most $d+1$ of these functionals nonzero on $x$, then the sum of Lipschitz constants of functionals centered at vertices of layer $N$, nonzero on $x$, is at most $(d+1)\frac{\sqrt{2}\cdot d(v,A) (z^2-1)}{z^2} \frac{z^{1-N}}{z-1}$. Furthermore, since $X$ is compact, there exists some $\tilde{M}$ such that $d(v,A) \leq \tilde{M}$ for all $v \in V$. Therefore, summing over $N$ gives us an upper bound $M$ on the the sum of Lipschitz constants of functionals which are nonzero on $x$. 

            \begin{equation*}
                \begin{split}
                   M&= \sum\limits_{N=0}^\infty (d+1) \frac{\sqrt{2}\cdot \tilde{M} (z^2-1)}{z^2} \frac{z^{1-N}}{z-1}\\
                   &=\frac{(d+1) z\sqrt{2}\cdot \tilde{M} (z^2-1)}{z^2(z-1)} \sum\limits_{N=0}^\infty \frac{1}{z^N}\\
                    &= \frac{(d+1) z^2\sqrt{2}\tilde{M}(z^2-1)}{z^2(z-1)^2}
                \end{split}
            \end{equation*}
            $$$$

\end{proof}

\section{Vectorizing Persistence Diagrams by Schauder Bases}

We now turn our attention to the stated goal of this paper, mapping signed persistence diagrams on polyhedral pairs into sequence space. We initially embed persistence diagrams into $\ell_1$ through a Lipschitz mapping. However, if one wishes to make use of a Hilbert space structure, one may compose this embedding with the embedding of $\ell_1$ into $\ell_2$, or into $\ell_p$ for $p>2$.

\subsection{Feature Maps for Signed Persistence Diagrams}

\begin{definition}
    Let $\mathbb{B}$ be a  Schauder basis of $(Lip_c(X,A), \ell_\infty)$. We will define \emph{vectorization by $\mathbb{B}$} to be the function $F_\mathbb{B}: D(X,A) \rightarrow \mathbb{R}^\omega$ mapping a persistence diagram $\alpha$ , to the sequence of real numbers given by

    $$F_\mathbb{B}(\alpha) :=\left(\alpha(f)\right)_{f \in \mathbb{B}}$$

    Where if $\alpha = \sum\limits_{i=0}^\infty sign(x_i) \cdot x_i$, then $\alpha(f) = \sum\limits_{i=0}^\infty sign(x_i) \cdot f(x_i)$. 
\end{definition}

Note that $\forall f \in \mathbb{B}$, $f \in Lip_c(X,A)$. Thus all such $f$ are $0$ on $A$. Hence the choice of a different representative of $\alpha$ with more or fewer terms from $A$ does not affect the value of the vectorization on $\alpha$. Hence $F_{\mathbb{B}}$ is well defined. 

\begin{theorem}[Stability Theorem]\label{thrm: schauder stability}
    Let $\mathbb{B}=$ be a totally ordered Schauder basis of $(Lip_c(X,A), \ell_\infty)$, with the locally Lipschitz finite property with upper Lipschitz bound $M$. Then vectorization by $\mathbb{B}$ is injective and

    $$||F_\B(\alpha) - F_\B(\beta)||_1 \leq 2M \cdot W_1(\alpha, \beta)$$
\end{theorem}

\begin{proof}

We first prove injectivity. Let $\alpha \in D(X,A)$. Note by definition of vectorization, $F_\B$ is linear on $D(X,A)$. Thus proving injectivity is equivalent to proving that if $F_\B(\alpha) =0$ then $\alpha =\emptyset = A$. 

Thus suppose that $\alpha \neq \emptyset$. Then $\alpha$ must include a point $x \in X \backslash A$ such that the multiplicity of $x$ in $\alpha$ is nonzero; $mult_\alpha(x) \neq 0$. 

Since $x \in X \backslash A$, and $A$ is closed,  there exists $\epsilon >0$ such that $B_{2\epsilon}(x) \cap A =\emptyset$. Let $C= \{x'  \in B_\epsilon(X)  | \ x'\neq x \ \text{ and }  \ mult_\alpha(x') \neq 0 \}$. By the triangle inequality, $d(x',A) >\epsilon $ for all $x' \in C$. Thus $|C|< \infty$ by supposition that $\alpha \in D(X,A)$. Hence, we may choose $\delta $ such that $0< \delta \leq \epsilon$ and  $B_\delta(x) \cap C = \emptyset$.

Define $g: (X,A) \rightarrow (\mathbb{R},0)$ to be the $1-$Lipschitz map $g(y) = \max\{0, \delta -d(x,y)\}$. 

By assumption that $\B$ is a Schauder Basis, there exists unique scalars $\{a_f\}_{f \in \B}$ such that $\sum\limits_{f \in \B} a_f \cdot f =g$. Note that $\alpha(g) = g(x) = \delta \cdot mult_\alpha(x)$. Therefore, $\sum\limits_{f \in \B}^\infty a_f \cdot \alpha(f) = \delta \cdot mult_\alpha(x)\neq 0$. Hence there must exist at least one function $f^* \in \B$ such that $ \alpha (f^*) \neq 0$. Hence $F_\B$ is injective.





    




    We now prove stability. We begin with single point pairings, then extend linearly. Let $x,y \in X$.  \\

Let $B_{x,y}= B_x \cup B_y$ be the set of all functionals $f \in \B$ which are nonzero on $x$ or $y$.  
    Then 
    \begin{equation*}
        \begin{split}
            \sum\limits_{f \in \mathbb{B}}|f(x)-f(y)| &= \sum\limits_{f \in B_{x,y}} |f(x)-f(y)|\\
            &\leq \sum\limits_{f \in B_{x,y}} d(x,y)\cdot Lip(f)\\
            &\leq 2M \cdot d(x,y)\\
        \end{split}
    \end{equation*}

    Now consider persistence diagrams of only positive terms,  $\alpha = \sum\limits_{i=0}^\infty x_i $ and $\beta = \sum\limits_{i=0}^\infty y_i$. Let $\sigma$ be a partial matching of $\alpha$ and $\beta$. 

    \begin{equation*}
        \begin{split}
            \sum\limits_{f \in \mathbb{B}} |\alpha(f)-\beta(f)| &= \sum\limits_{f \in \mathbb{B}} |\sum\limits_{i=0}^\infty f(x_i) -f(y_{\sigma(i)})|\\
            &\leq \sum\limits_{f \in \mathbb{B}} \sum\limits_{i=0}^\infty |f(x_i) -f(y_{\sigma(i)})|\\
            &= \sum\limits_{i=0}^\infty \sum\limits_{f \in \mathbb{B}} |f(x_i) -f(y_{\sigma(i)})|\\
            &\leq \sum\limits_{i=0}^\infty 2M\cdot d(x_i, y_{\sigma(i)})\\
            &= 2M\cdot Cost(\sigma)\\
        \end{split}
    \end{equation*}

    Taking infemum over all partial matchings $\sigma $ gives us the stability inequality 

    $$||F_\mathbb{B}(\alpha) - F_\mathbb{B}(\beta)||_1 \leq 2M \cdot  W_1(\alpha, \beta)$$\\

    Finally, consider the case that $\alpha = (\alpha^+, \alpha^-)$ and $\beta=(\beta^+, \beta^-)$ might have negative terms. Then 

    \begin{equation*}
        \begin{split}
            ||F_\B(\alpha) - F_\B(\beta)||_1 &= ||(F_\B(\alpha^+) - F_\B(\alpha^-)) -(F_\B(\beta^+) - F_\B(\beta^-))||_1\\
            &= ||F_\B(\alpha^+ + \beta^-) - F_\B(\beta^++ \alpha^-)||_1 \\
            &\leq 2M\cdot W_1(\alpha^+ + \beta^-, \beta^+ + \alpha^-)\\
            &= 2M \cdot W_1(\alpha, \beta)
        \end{split}
    \end{equation*}

\end{proof}

\begin{corollary}
    Considering the case that $\beta = \emptyset$, and using the fact that $\alpha \in D(X,A)$, we determine that if $\B$ meets the locally Lipschitz finite property, then $F_\B$ is an embedding of $D_{\pm} (X,A)$ into sequence space $\ell_1$. 
\end{corollary}

 Theorem \ref{thrm: schauder stability} applies to any Schauder basis meeting the locally lipschitz finite property. By Lemma \ref{lemma: LLF}, this  includes  Schauder bases defined using  nested triangulations as in Definition \ref{def:kernel}, which meets the l.l.f. property with $M= L \cdot(d+1)$. In the case that this nested triangulation is a nested Coxeter-Freudenthal Kuhn triangulation (as in Example \ref{example:CFK triangulation} and Example \ref{example: CFK triangulation R3}), the stability bound of Theorem \ref{thrm: schauder stability} is not optimal. We will illustrate the optimal bound for this setting in the following. \\






\begin{lemma}\label{lemma: CFK Lipschitzness}
    Let $d>1$ and $(X,A)$ be a polyhedral pair in $\mathbb{R}^d$ defined by inequatlities of the form $x_i \leq x_j$ for some pairings of $i \leq j$. Let $(T,S)  $ be the  Coxeter-Freudenthal-Kuhn triangulation at scale $c$ on $(X,A)$. Let $V$ be the vertices of $T$ that are not in $S$, and $B = \{\K_v\}_{v \in V}$ be the collection of all p.l. functionals on $(T,S)$ as defined in definition \ref{def:kernel}, with common Lipschitz constant $L>0$ for all $v \in V$. Let $V'$ be any subset of the vertex set $V$ of $T$, and let $f = \sum\limits_{v \in V' } \K_v$. Then $Lip(f) \leq \sqrt{\frac{d}{2}} \cdot L$. 
\end{lemma}

\begin{proof}
    
    By scaling, we may assume $c=1$. We prove this lemma by restricting to a single simplex $\sigma$ of $T$. Since $X$ is convex and contains all its geodesics, an upper bound on the Lipschitz constant of $Lip(f|_\sigma)$ induces an upper bound of $Lip(f)$.  Since each $d-$simplex of $T$ is congruent to the standard $d-$simplex, we may assume without loss of generality that $\sigma = \langle v_0, v_1,...v_d\rangle$ where $v_i = \sum\limits_{j=1}^i e_i$ for $i >0$, and $v_0 =0$.\\

    We will first calculate the distance of each vertex to its opposite hyperplane.

    The face of  $\sigma$ opposite $v_0$ lies in the hyperplane $H_0$ spanned by vectors $\{v_i-v_1 = \sum\limits_{i=2}^k e_i\}_{k =2}^d$, which has normal vector $N_0:=e_1$. Then since $e_1$ lies on $H_0$, the distance of the vertex $v_0$ to $H_0$ is given by $$d(v_0, H_0) = \frac{|N_0 \cdot (v_0 -e_1)|}{||N_0||}=1$$

    For $0 <p<d$, the face opposite $v_p$ lies in the hyperplanes $H_p $ spanned by vectors $\{v_i= \sum\limits_{j=1}^i e_j \}_{i \neq p}$, which has normal vector $N_p: = e_{p+1}-e_p$. Thus since $v_0=0$ lies on $H_p$,  the  vertex $v_p = \sum\limits_{i=1}^p e_p$ has distance to $H_p$ given by 
    $$d(v_p, H_p) = \frac{|N_p \cdot (v_p-0)|}{||N_p||}= \frac{1}{\sqrt{2}}$$

    The face opposite $v_d$ lies in the hyperplane $H_d$ spanned by vectors $\{v_i = \sum\limits_{j=1}^i e_j\}_{i <d}$. $H_d$ has normal vector $N_d := -e_d$. Then since $v_0=0$ lies in $H_d$, the distance of $v_d$ to $H_d$ is given by $$d(v_d, H_d) = \frac{|N_d \cdot (v_d - 0)|}{||N_d||}=1$$ 

    Therefore, the minimal distance of a vertex to the hyperplane defined by its opposite face is given by $\frac{1}{\sqrt{2}}$. \\

    Recall that to define functionals $\K_{v_0}$ and $\K_{v_d}$, we take the minimal distance of the vertices $v_0$ and $v_d$ to their opposite face hyperplanes over all simplices containing each vertex respectively. \\

   We show now that there exists simplices $\sigma^0, \sigma^d$ in $T$ containing $v_0$ and $v_d$  respectively,  such that the distance from $v_0,v_d$ to their opposite hyperplanes in $\sigma^0 , \sigma^d$ is  $\frac{1}{\sqrt{2}}$.\\

    For $v_0=0$, consider the simplex defined by the translation of the previous $\sigma$ by the vector $-e_1$;  $\sigma^0 : = \langle -e_1, 0, e_2, e_2+e_3, ...v_d-e_1\rangle$. Let $h: \sigma \mapsto \sigma^0$ be the translation map. Consider an inequality $x_i \leq x_j$ defining the polyhedron $X$. If $i \neq 1$, then for any point $x \in \sigma$, $h(x)_i = x_i \leq x_j = h(x)_j$. If $i=1$, then $h(x)_i = x_i-1 < x_i \leq x_j = h(x)_j$. Therefore, the simplex $\sigma^0 \in T$. Furthermore, the vertex $0$ has opposite face lying in the hyperplane $H_0'$ spanned $\{\sum\limits_{j=1}^i e_i \}_{i=2}^d$. Therefore, a normal vector to the hyperplane $H_0'$ is $N_0' := e_1-e_2$. Since $-e_1$ is on $H_0'$, we may then determine that the distance of $v_0=0$ to the hyperplane $H_0'$ is given by 
    $$d(0, H'_0) = \frac{|N_0' \cdot (0+e_1)|}{||N_0'||}= \frac{1}{\sqrt{2}}$$

    Similarly, we may translate $\sigma $ by $e_d$ to find a simplex $\sigma^d$ containing the vertex $v_d= \sum\limits_{i=0}^d e_i$, such that the distance of $v_d$ to the hyperplane containing its opposite face in $\sigma^d$ is $\frac{1}{\sqrt{2}}$. \\

    Therefore, all vertices of $\sigma $ have equal minimum distance to a hyperplane containing their opposite faces in the triangulation $T$. Thus, since all $\K_{v_i}$ are specified to have the same Lipschitz constant $L$, using the proof of Proposition \ref{prop:kernels exist}, the value $\K_{v_i}(v_i) = \frac{L}{\sqrt{2}}$ for all $i$. \\

    Furthermore, we obtain the gradients of each functional by the following. For each $i$, let $\bar{\K}_{v_i}: X \rightarrow \mathbb{R}$ be the linear extension of $\K_{v_i} |_{\sigma}$. Then  $\nabla(\bar{\K}_{v_0})= \frac{L}{\sqrt{2}} (-e_1)$. Additionally, for $0 <p<d$,  $\nabla(\bar{\K}_{v_i}) = \frac{L}{\sqrt{2}}\cdot (e_{p+1}-e_p)$ , and   $\nabla(\bar{\K}_{v_d}) = \frac{L}{\sqrt{2}} e_d$. \\

    Note that $\sum\limits_{p=0}^d \nabla (\bar{\K}_{p})=0$. \\

    Let $I \subset \{0,1,...d\}$, and $f = \sum\limits_{i \in I} \bar{\K}_{v_i}$. \\
    

    Fix $i$ and first consider the case that  $i,i+1 \in I$. Then the $i+1$th coordinate of $\nabla(f) $ is $0$. However, if $i \in I$ and $ i+1 \notin I$, then the $i+1$th coordinate of $\nabla(f)$ is $\frac{L}{\sqrt{2}}$. Lastly, if $i \notin I$ and $i+1 \in I$, then the $i+1$th coordinate of $\nabla(f)$ is $-\frac{L}{\sqrt{2}}$. \\

    Thus, in the worst case scenario, all $d$ components of $\nabla(f)$ are $\pm \frac{L}{\sqrt{2}}$.
    
    Thus $||\nabla(f) || \leq \sqrt{d \cdot \left(\frac{L}{\sqrt{2}}\right)^2}= \sqrt{\frac{d}{2}}\cdot L$. 

    Therefore, using the geodesics of the polyhedron $X$, we may extend this result over all simplices of $T$, and conclude that for a subset $V' \subset V$,

    $$Lip\left(\sum\limits_{v \in V'} \K_{v'}\right) \leq \sqrt{\frac{d}{2}  } \cdot L$$

\end{proof}

We now use this lemma to prove an optimal stability result for the 1-Wasserstein distance on persistence diagrams.

\begin{theorem}[CFK-Stability]\label{thrm:optimal bound}
    Let $\Upsilon \subset \{ (i,j) | \ 1 \leq i < j \leq d\}$ be a set of relations on indices of $\mathbb{R}^d$. Let $X \subset \mathbb{R}^d$ be the subset of $\mathbb{R}^d$ such that $X = \{x \ | \ x_i \leq x_j \ \forall (i,j) \in \Upsilon\}$.  Furthermore, let  $\Upsilon ' \subset \Upsilon$ be a nonempty subset of relations. Let $A = \{x \in \mathbb{R}^d \ | \exists (i,j) \in \Upsilon ' \ s.t.  \ x_i = x_j \}$. Let ${(T^n,S^n)}_{n=0}^\infty$ be the nested triangulation, where $(T^n, S^n)$ is the Coxeter-Freudenthal-Kuhn triangulation, scaled by $\frac{1}{z^n}$ for some integer $z>1$. Let $\mathbb{B}$ be the Schauder basis of $Lip_c(X,A)$ of functionals as in Definition \ref{thrm: Schauder Basis}, with Lipschitz constants $(L_n)_{n=0}^\infty$, $\sum_{n=0}^\infty L_n =L$. Then for any two persistence diagrams $\alpha, \beta \in D(X,A)$, 
    $$||F_\B(\alpha) - F_\B(\beta)||_1 \leq \sqrt{2d}\cdot L \cdot  W_1(\alpha, \beta)$$

    Furthermore, this bound is optimal for all choices of $\{L_n\}_{n=0}^\infty$ and dimensions $d$. That is, for all $\epsilon >0$, there exists $d>0$ and a choice of $\{L_n\}_{n=0}^\infty $ , and persistence diagrams $\alpha, \beta \in D_{\pm}(X,A)$ such that $$||F_\B(\alpha) - F_\B(\beta)||_1 >W_1(\alpha, \beta) - \epsilon  $$

    \end{theorem}

\begin{proof}
As before, we begin with a proof for persistence diagrams consisting of at most a single point. \\

Let $x,y \in X$ and $\alpha = x$, $\beta =y$. Then 

\begin{equation*}
    \begin{split}
        ||F_\B(\alpha) - F_\B(\beta)||_1 &= \sum\limits_{\K_v \in \B} |\K_v(x)  - \K_v(y)|\\
        &= \sum\limits_{n=0}^\infty \sum\limits_{v \in V^{(n)}} |\K_v(x) - \K_v(y)|\\
    \end{split}
\end{equation*}

For each $n$, let $A_n= \{v \in V^{(n)} \ | \ \K_v(x) \geq \K_v(y)\}$ and let $B_n= \{v \in V^{(n)} \ | \ \K_v(x) < \K_v(y)\}$. Then the above becomes, \\

\begin{equation*}
    \begin{split}
        \sum\limits_{n=0}^\infty \sum\limits_{v \in V^{(n)}} |\K_v(x) - \K_v(y)| &= \sum\limits_{n=0}^\infty\sum\limits_{v \in A_n } (\K_v(x) - \K_v(y)) + \sum\limits_{v \in B_n} (\K_v(y) - \K_v(x))\\
        &= \sum\limits_{n=0}^\infty \left(\left(\sum\limits_{v \in A_n} \K_v\right)(x) - \left(\sum\limits_{v \in A_n} \K_v\right)(y)\right) + \left(\left(\sum\limits_{v \in B_n} \K_v\right)(y) - \left(\sum\limits_{v \in B_n} \K_v\right)(x)\right)\\
        &\leq \sum\limits_{n=0}^\infty 2 \cdot\sqrt{\frac{d}{2}}\cdot  L_n \cdot d(x,y) \hspace{1cm} \textbf{ (by Lemma \ref{lemma: CFK Lipschitzness})}\\
        &= \sqrt{2d}\cdot L \cdot d(x,y)
    \end{split}
\end{equation*}
\vspace{0.5cm}

Now let us extend to other nonnegative persistence diagrams, $\alpha = \sum\limits_{i=0}^\infty x_i, \beta = \sum\limits_{i=0}^\infty y_i \in D_+(X,A)$. Let $\sigma$ be a partial matching of $\alpha $ and $\beta$. \\

\begin{equation*}
    \begin{split}
        ||F(\alpha) - F(\beta)||_1 &= \sum\limits_{n=0}^\infty \sum\limits_{v \in V^{(n)} } |\alpha(\K_v)- \beta(\K_v)| \\
        &=  \sum\limits_{n=0}^\infty \sum\limits_{v \in V^{(n)} } | \sum\limits_{i=0}^\infty \K_v(x_i) - \K_v(y_{\sigma(i)})| \\
        &\leq \sum\limits_{i=0}^\infty \sum\limits_{n=0}^\infty \sum\limits_{v \in V^{(n)} } |\K_v(x_i) - \K_v(y_{\sigma(i)})| \\
        &\leq \sum\limits_{i=0}^\infty \sqrt{2d} L \cdot d(x_i, y_{\sigma(i)}) \\
        &= \sqrt{2d}L\cdot Cost(\sigma)
    \end{split}
\end{equation*}

Taking infemum over all partial matchings, this gives us the desired inequality

$$||F(\alpha) - F(\beta)||_1 \leq \sqrt{2d}\cdot L \cdot W_1(\alpha, \beta)$$\\

Finally, consider the case that $\alpha = (\alpha^+, \alpha^-)$ and $\beta=(\beta^+, \beta^-)$ might have negative terms. Then 

    \begin{equation*}
        \begin{split}
            ||F_\B(\alpha) - F_\B(\beta)||_1 &= ||F_\B(\alpha^+) - F_\B(\alpha^-) -F_\B(\beta^+) + F_\B(\beta^-)||_1\\
            &= ||F_\B(\alpha^+ + \beta^-) - F_\B(\beta^++ \alpha^-)||_1 \\
            &\leq \sqrt{2d} \cdot L \cdot W_1(\alpha^+ + \beta^-, \beta^+ + \alpha^-)\\
            &= \sqrt{2d}  L \cdot W_1(\alpha, \beta)
        \end{split}
    \end{equation*}\\

We now show this bound is the optimal bound to hold for all choices of Lipschitz constants for any dimension $d$. \\

Let $d$ be even. Let $\epsilon >0$ and $L>\epsilon$. Define $L_0 = L-\epsilon$ and $L_n = \frac{\epsilon}{2^n}$ for all $n >0$. 

Let $\sigma \in T^0$ be a simplex such that all vertices of $\sigma$ are not in $A$. For simplicity of notation, we will denote $\sigma = \langle v_0, v_1, ...v_d\rangle$ where $v_0=0$ and $v_p = \sum\limits_{i=0}^p e_i$ for all $p >0$, where $v_p \in V^{(0)}$ for all $p$. 

By proof of Lemma \ref{lemma: CFK Lipschitzness}, we know that the gradient vector of $\K_{v_p}|_\sigma$ is given by $\frac{L_0}{\sqrt{2}}N_p$ where $N_0=e_1, $  $N_p = e_{p+1}-e_p$ for $1\leq p <d$, and $N_d = -e_d$. 

Choose $x,y \in \sigma$ such that the vector $y-x$ is a positive scalar multiple of the vector $N_{odd} := \sum\limits_{p \ \text{ odd}} N_p $. 
Since $\sum\limits_{p=0}^d \frac{L_0}{\sqrt{2}}N_p =0$, this implies that $x-y$ is a positive scalar multiple of $N_{even} := \frac{L_0}{\sqrt{2}}\sum\limits_{p \ \text{ even}} N_p$. Note then that $\K_{v_p}(y) \geq \K_{v_p}(x) $ iff $p$ is odd. Furthermore, $Lip\left( \sum\limits_{i \ odd} \K_{v_p}|_{\sigma}\right)= ||N_{odd}||= \sqrt{\frac{d}{2}} L_0$.  Therefore, 

\begin{equation*}
    \begin{split}
        \sum\limits_{i=0}^d |\K_{v_i}(y) - \K_{v_i}(x)|&= \left( \sum\limits_{i \ odd} \K_{v_i}(y) - \K_{v_i}(x)\right) + \left(  \sum\limits_{i \ even} \K_{v_i}(x) - \K_{v_i}(y)\right)\\
        &= 2 \cdot \sqrt{\frac{d}{2}} L_0 \cdot d(x, y)\\
        &= \sqrt{2d} L_0 \cdot d(x, y)
    \end{split}
\end{equation*}

Chossing $\alpha =x$ and $\beta=y$ then gives us that 

\begin{equation*}
    \begin{split}
        ||F_\B(\alpha) - F_\B(\beta)||_1 &=  \sum\limits_{n=0}^\infty \sum\limits_{v \in V^{(n)}} |\alpha(\K_v) - \beta(\K_v)|\\
        & \geq\sum\limits_{v \in V^{(0)}} |\K_{v}(x) - \K_{v}(y)|\\
        &= \sqrt{2d} L_0 \cdot W_1(\alpha, \beta) \\
        &= \sqrt{2d}L \cdot W_1(\alpha, \beta) - \sqrt{2d} \cdot \epsilon \cdot W_1(\alpha , \beta)
    \end{split}
\end{equation*}

\end{proof}
\begin{corollary}

    Fix some $N \geq 0$. Let $F_\mathbb{B}^N$ be the vectorization defined   by the restriction of the vectorization $F_\mathbb{B}$ to p.l. functionals  centered at vertices in $ \bigcup\limits_{n=0}^N V^{(n)}$ of a nested triangulation on polyhedral pair $(X,A)$. Denote the Lipschitz constants of functionals at each layer by $\{L_n\}_{n=0}^\infty$. Then for any persistence diagram, $||F_\B(\alpha ) - F_\B^N(\alpha)||_1 \leq (d+1) \sum\limits_{n=N+1}^\infty  L_n\cdot W_1(\alpha, \emptyset)$. In the case that the triangulation is a CFK triangulation,   $||F_\B(\alpha ) - F_\B^N(\alpha)||_1 \leq \sqrt{2d} \sum\limits_{n=N+1}^\infty  L_n\cdot W_1(\alpha, \emptyset)$. That is, the vectorization map $F_\B$ may be approximated to arbitrary accuracy using a finite number of layers of the nested triangulation.\\ 
\end{corollary}

The proof of this follows directly from the proofs of Theorem \ref{thrm: schauder stability} and Theorem \ref{thrm:optimal bound}.\\

We now recall the Schauder basis of \emph{Stacked Functionals} as described in Definition \ref{def:stacked functional}. This Schauder basis has the appealing property that vectorization by this basis maps unsigned persistence diagrams to to a vector with a 1-norm equal to the 1-Wasserstein distance of the diagram to the empty diagram. 


\begin{proposition}\label{thrm:stacked schauder eval}
    Let $X\subset \mathbb{R}^d$ be a compact polyhedron and suppose polyhedral pair $(X,A)$ is endowed with the nested CFK-triangulation $\{(T^n, S^n)\}_{n=0}^\infty$ of scales $\frac{1}{z^n}$ on a pair $(X,A)$.  Let $\B$ be the Schauder basis of stacked functionals as in Definition \ref{def:stacked functional}. Then for any $\alpha \in D_+(X,A)$, $||F_\B(\alpha)||_1 = W_1(\alpha, \emptyset)$. 
\end{proposition}

\begin{proof}

    Since we are restricting to unsigned persistence diagrams, it is sufficient to prove that $\sum\limits_{v \in V} \mathfrak{K}_v = d(-, A)$. To verify this, we will first show that $d(-,A)$ is $N-$linear for all $N \geq 0$. By Lemma \ref{lem:n2n+1linear}, it is sufficient to prove this for $N=0$. \\

    Recall that there exists a collection of relations $\Upsilon\subset \{(i,j)\ | \ i <j\leq d\}$ on coordinates of points in $X \subset \mathbb{R}^d$, which define $X$ as a polyhedron in $\mathbb{R}^d$; i.e. $X = \{x \in \mathbb{R}^d \ | \ x_i \leq x_j \ \forall (i,j) \in \Upsilon\}$. Let $\Upsilon'\subset \Upsilon$ be the nonempty subset of relations defining $A \subset X$; that is $A= \bigcup\limits_{(i,j) \in \Upsilon'} \{x \in X \ | \ x_i =x_j\}$.

    Let $x \in X$. For $(i,j) \in \Upsilon'$, consider the closed subspace $A_{(i,j)} = \{x \in X \ | \ x_i =x_j\}$. Then $d(x,A_{i,j}) = \sqrt{\left| x_j-\frac{x_j+x_i}{2}\right|^2 + \left|\frac{x_j+x_i}{2}-x_i\right|^2 }= \frac{1}{\sqrt{2}}(x_j-x_i)$. Thus $d(-,A_{i,j})$ is in fact linear with respect to any triangulation, including the CFK-triangulation at scale 1. \\

    Then note that $A = \bigcup\limits_{(i,j) \in \Upsilon'} A_{(i,j)}$. Therefore, $D(x,A) = \min\limits_{(i,j) \in \Upsilon'} d(x, A_{i,j})$. We claim that for any simplex $\sigma\in T^0$, there exists $(i,j)\in \Upsilon'$ such that $d(v,A) = d(v, A_{(i,j)})$ for all vertices $v \in \sigma$. To see this, $(i,j) , (i',j') \in \Upsilon'$ such that $(i,j) \neq (i',j')$. Then consider the subspaces where $d(x, A_{(i,j)}) = d(x, A_{(i',j')})$. Equivelently, this space is defined by the set of all $x \in X$ such that $x_j-x_i = x_{j'} - x_{i'}$. When restricting $x$ to be vertices $v$ of $T^0$, the coordinates of $v$, denoted $\{(v)_k\}_{k=1}^d$, are integers. Therefore, the hyperplanes $\{x \ | \ x_j-x_i = (v)_{j'} - (v)_{i'} \}_{(v \in V^0 \ | \ (i',j' ) \in \Upsilon')}\subset \{x \ | \ x_j-x_i = z \}_{z \in \mathbb{Z}}$. These are the very hyperplanes that define faces of d-simplices in the CFK-triangulation at scale $1$. Therefore, all vertices $v$ of $\sigma$ must exist within one side of each one of these hyperplanes. Hence there is indeed some $(i,j)$ for which $d(v,A_{i,j}) = d(v,A)$ for all $v \in \sigma$. Hence $d(-,A)$ is in fact linear on $T^0$, and thus also on $T^N$ for all $N \geq 0$.

    We may now prove that $\sum\limits_{v \in V} \mathfrak{K}_v = d(-, A)$. Recall by definition that 

    $$\mathfrak{K}_v:= \frac{\sqrt{2}\cdot d(v,A) (z^2-1)}{z^2}\sum\limits_{n=N}^\infty \K_v^n$$

    Fix $v^* \in V^{(N)}$. Recall that  $\K_v^n(v^*) = \frac{1}{z^{2n}\sqrt{2}}$ for all $n \geq N$. Therefore,

    \begin{equation*}
        \begin{split}
            \sum\limits_{v \in V} \mathfrak{K}_v&= \sum\limits_{N=0}^\infty \sum\limits_{v \in V^{(N)}} \mathfrak{K}_v \\
            &= \sum\limits_{N=0}^\infty \sum\limits_{v \in V^{(N)}} \frac{\sqrt{2}\cdot d(v,A) (z^2-1)}{z^2}\sum\limits_{n=N}^\infty \K_v^n\\
            &=  \sum\limits_{n=0}^\infty \sum\limits_{v \in V^{n}} \frac{\sqrt{2}\cdot d(v,A) (z^2-1)}{z^2} \K_v^n\\
        \end{split}
    \end{equation*}

    For $v^* \in V^n$, $\K_{v^*}^n(v^*) = \frac{1}{z^{2n}\sqrt{2}}$. Furthermore, $\K_{v'}^n(v^*) =0 $ for all $v' \in V^n$ such that $v' \neq v^*$. Therefore, $\sum\limits_{v \in V^{n}} \frac{\sqrt{2}\cdot d(v,A) (z^2-1)}{z^2} \K_v^n(v^*) =  \frac{\sqrt{2}\cdot d(v^*,A) (z^2-1)}{z^2} \K_{v^*}^n(v^*)= \frac{(z^2-1)}{z^{2n+2}} d(v^*,A)$. 

    Since this is for all $v^* \in V^n$, and by the fact that $d(-,A)$ is $n-$linear, we then infer that $\sum\limits_{v \in V^{n}} \frac{\sqrt{2}\cdot d(v,A) (z^2-1)}{z^2} \K_v^n =\frac{(z^2-1)}{z^{2n+2}} d(-,A) $. \\

    Summing over all layers of the triangulation, we gather that $\sum\limits_{v \in V} \mathfrak{K}_v= \sum\limits_{n=0}^\infty \frac{(z^2-1)}{z^{2n+2}} d(-,A) = d(-,A)$.

\end{proof}

\subsection{Visualizing the Feature Map on Signed Persistence Diagrams}

The result of the vectorization $F_\B$ applied to a persistence diagram $\alpha= \sum\limits_{i=0}^\infty x_i$ is a vector $F_\B(\alpha)$ in $\ell_1$. We could attempt to visualize this by coloring segments of $\mathbb{R}_{\geq 0}$, where each segment is proportional to entries in $F_\B(\alpha)$.   That is, if we denote the components of $F_\B(\alpha)$ by $\{a_i\}_{i=0}^\infty$, then we partition $\mathbb{R}_{\geq 0}$ into intervals of the form $[\sum\limits_{i=-1}^{n-1} |a_i|, \sum\limits_{i=-1}^n |a_i|)$, where $a_{-1}:=0$. Since $F_\B(\alpha) \in \ell_1$, $\sum\limits_{i=0}^\infty |a_i| =:a < \infty$. Thus this collection of intervals will only cover the line segment $[0, a)$. Assign a color to each interval in the partition of $[0,a)$. This colored line segment may be the most direct visual representation of our vectorization, however we do not feel that it portrays the properties of our method in an intuitive manner. We instead, decompose $F_\B(\alpha)$ into its component vectors. Recall that $F_\B$ is linear, and thus $F_\B(\alpha) = \sum\limits_{i=0}^\infty F_\B(x_i)$. We then create the color-coded line segment, as described above, for each point $x_i$ of the persistence diagram. In the case that $\alpha$ is a persistence diagram of a 1-parameter persistence module, we plot this collection of color-coded line segments in $\mathbb{R}^3$. Here each line segment is plotted parallel to the $z-$axis, atop the corresponding point $x_i$ in the persistence diagram, which is plotted in the $xy-$plane.

\begin{figure}[H]
\begin{tikzpicture}
    \node (PD) at (0,0){\includegraphics[width=0.3\linewidth, trim =25pt 130pt 650pt 150pt, clip]{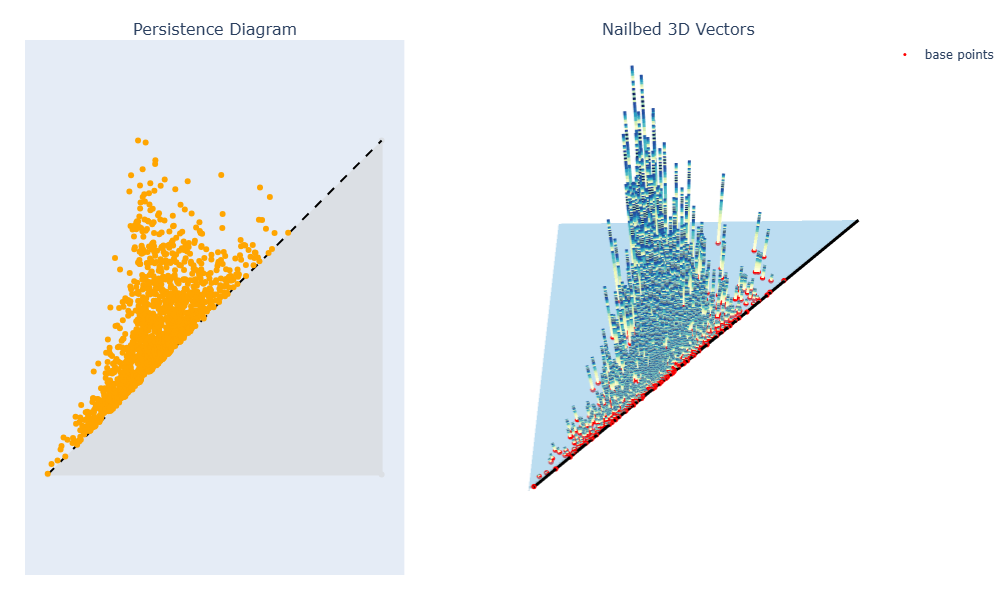}};
    \node (TOP) at (6,0){  \includegraphics[width=0.44\linewidth, trim =450pt 45pt 75pt 50pt, clip]{PD-top.png}};
    \node (TOP) at (0,-6){  \includegraphics[width=0.4\linewidth, trim =450pt 45pt 75pt 50pt, clip]{PD-topside.png}};
    \node (TOP) at (6,-6){  \includegraphics[width=0.4\linewidth, trim =450pt 45pt 75pt 50pt, clip]{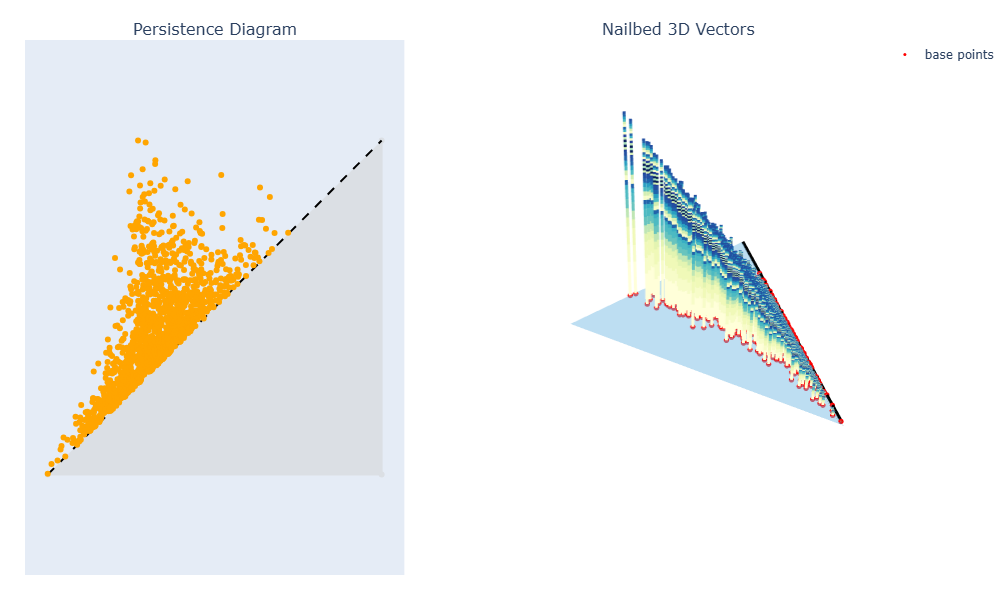}};
\end{tikzpicture}
   
    \caption{Visualization of the vector $F_\B(\alpha)$ on a persistence diagram of a 1-parameter persistence module. Here $F_\B$ is applied to each point of the persistence diagram, yielding a vector in $\ell_1$ for each point. These vectors are plotted as multicolored line segments. }
    \label{fig:nailbed}
\end{figure}

In the case that $\alpha$ is a signed barcode froma 2-parameter persistence module, then we may employ a similar decomposition of $F_\B(\alpha)$ into its component vectors over points of the persistence diagram. However, to do so as in the case of the 1-parameter case, we would need to plot colored line segments in $\mathbb{R}^5$. To remedy this, we instead use signed line segments in $\mathbb{R}^2$ to represent rectangles of the signed barcode, as in the setting of Botnan, Opperman, and Oudot \cite{signedbarcodes}. In this setting, a rectangle $R$ with $\inf(R)=a$ and $\sup(R)=b$, the rectangle $R$ is represented by the line segment from point $a$ to $b$ in $\mathbb{R}^2$. This line segment is colored blue if $\langle a,b\rangle$ is a positive rectangle in the signed barcode, and red if it is negative. \\

Now, as in the case of persistence diagrams from 1-paramter persistence, we vectorize each bar of the signed barcode by $F_\B$. Instead of plotting line segments parallel to the z-axis, we plot sheets parallel to the z-axis, atop or below line segments in the signed barcode. We plot these sheets, or towers, in the positive z-direction if $\langle a,b \rangle$ is a positive rectangle in the barcode, and in the negative z-direction if $\langle a,b \rangle $ is negative in the barcode.

\begin{figure}[H]
    \begin{tikzpicture}
        \node(Barcode) at (0,0){\includegraphics[width=0.25\linewidth]{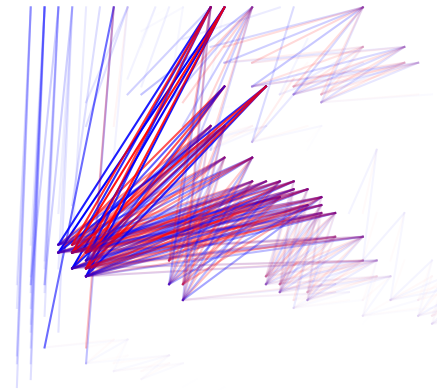}};
        \node (multipersatop) at (6,0){\includegraphics[width=0.35\linewidth, trim=600pt 0pt 30pt 30pt, clip]{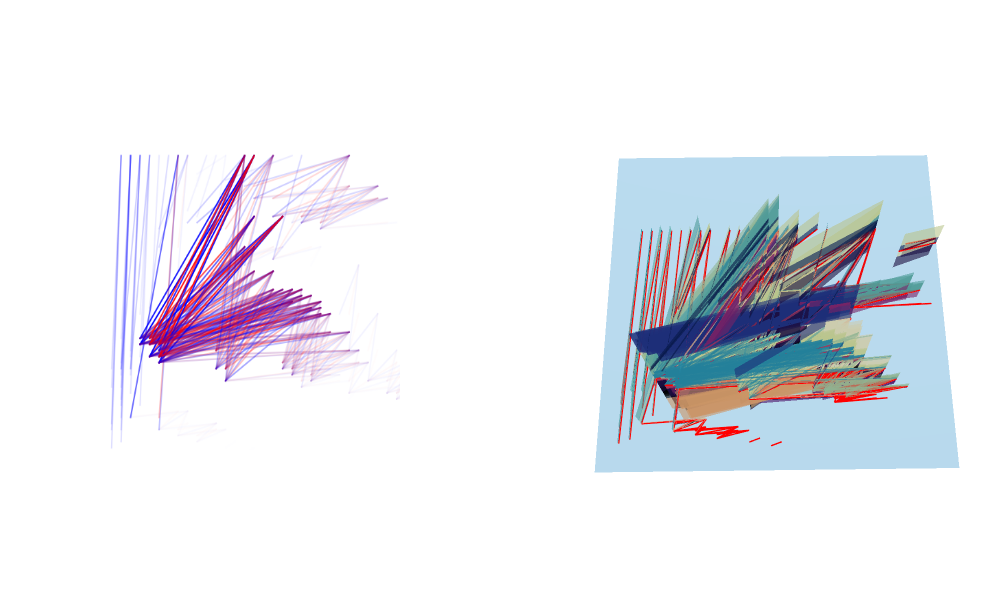}};
        \node (multiperunder) at (0,-6){\includegraphics[width=0.35\linewidth, trim=600pt 30pt 30pt 30pt, clip]{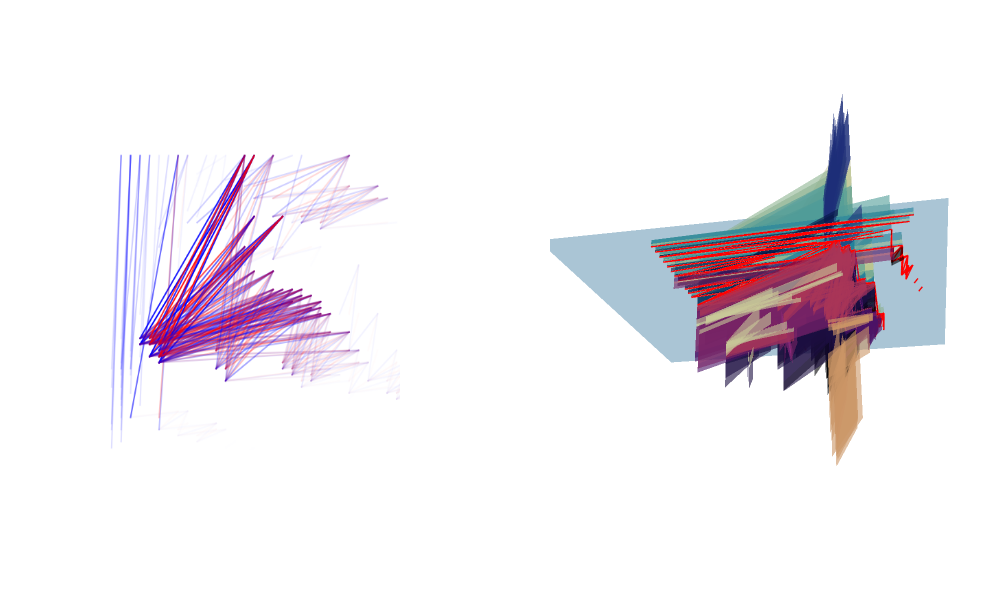}};
        \node (multiperside) at (6,-6){\includegraphics[width=0.35\linewidth, trim=600pt 30pt 30pt 30pt, clip]{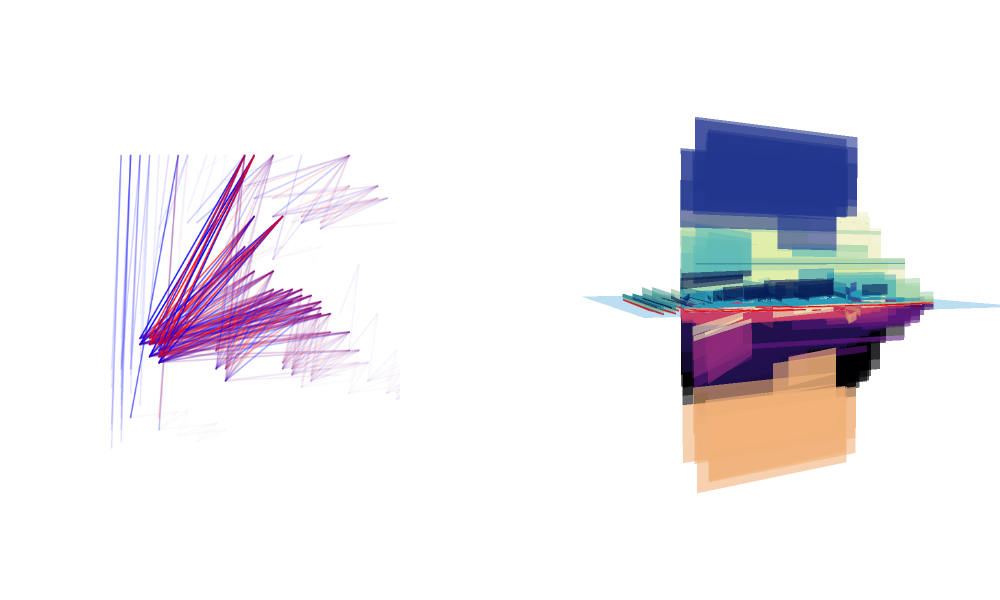}};
    \end{tikzpicture}
    \caption{Visualization of the vector $F_\B(\alpha)$ on a persistence diagram of a 2-parameter persistence module. The barcode, represented by a set of positive and negative line segments of positive slope, is generated with the Multipers code by Loiseaux et. al. \cite{multipers}. Here $F_\B$ is applied to each point of the corresponding signed persistence diagram, yielding a vector in $\ell_1$ for each diagram point. These vectors are plotted as multicolored towers atop the corresponding line segment if the line segment is positive, and below the line segment if the line segment is negative in the signed barcode.}
    \label{fig:multi-vis}
\end{figure}

We note in this visualization that bars of the barcode that have very small or very large slope do not have very tall towers on them. This is because all functionals of $\B$ are $0$ on the subspace $A$ of diagram space $X$. In this setting, $A$ is the space of points representing rectangles with infemum $a=(a_0,a_1)$ and supremum $b=(b_0, b_1)$, such that either $a_0=b_0$, or $a_1=b_1$. That is, $A$ consists of all ``flat rectangle''. These flat rectangles get represented as vertical or horizontal line segments in the visualization of the top left portion of Figure \ref{fig:multi-vis}. Since all functionals of $\B$ are $0$ on these flat rectangles, and the collection $\B$ is locally Lipschitz finite, then rectangles that are nearly flat, or close to $A$, do not contribute as much to the vectorization $F_\B(\alpha)$. \\

\section{Generalizing to Relative Radon Measures}
We now extend our results on vectorization of persistence diagrams on $(X,A)$ to relative Radon measures on $(X,A)$, as in the setting by Bubenik and Elchesen\cite{PeterAlex}. Some of the proofs of these results are analogous to the proofs in the previous section when extended to relative Radon measures. 

For Schauder basis $\B= \{f_i\}_{i=0}^\infty$ of $Lip_c(X,A)$, define vectorization of $ \hat{\mathcal{M}}_1(X,A)$ by $\B$ to be $F_\B: \hat{\mathcal{M}}_1(X,A) \rightarrow \ell_1$ by $F_\B(\alpha) = \left(\alpha(f_i)\right)_{i=0}^\infty = \left( \int\limits_X f_i \ d \alpha \right)_{i=0}^\infty $. Note that for each element $f_i$ of $\B$, $Lip(f_i) <\infty$ and $f_i(A)=0$. By definition of $\hat{M}_1(X,A)$, $\int_X d(-,A) \ d\alpha < \infty$. Hence $\int_X f_i \ d\alpha \leq \int_X Lip(f_i)\cdot  d(-,A) \ d\alpha < \infty$. Thus this is indeed a well defined map into $\mathbb{R}^\omega$.


We show that vectorization $F_\B$ of a $\hat{\M}(X,A)$ is stable and injective. First, we will need the following lemma. \\

\begin{lemma}
    Let $\B$ be a l.l.f. Schauder basis of $(Lip_c(X,A), ||\cdot ||_\infty)$, with local Lipschitz upper bound $M$. Let $\mathcal{A} \subseteq \B$. Then $g:=\sum\limits_{f \in \mathcal{A}} f$ has Lipschitz constant at most $2M$. 
\end{lemma}

\begin{proof}
    Let $x,y \in X$, and let $A_x, A_y$ be the sets of functions in $\mathcal{A}$ which are nonzero on $x$ and $y$ respectively. Then by assumption that $\B$ is l.l.f. with local Lipschitz upper bound $M$, we know $\sum\limits_{f \in A_x} Lip(f) , \sum\limits_{f \in A_y} Lip(f) \leq M$. Therefore, we have the following. 

    \begin{center}
        \begin{equation*}
            \begin{split}
                |g(x)-g(y)| & = \left| \sum\limits_{f \in A_x \cup A_y} f(x)-f(y)\right| \\
                &\leq \sum\limits_{f \in A_x \cup A_y} |f(x)-f(y)|\\
                &\leq \sum\limits_{f \in A_x \cup A_y} Lip(f) \cdot d(x,y)\\
                &\leq 2M \cdot d(x,y)
            \end{split}
        \end{equation*}
    \end{center}
\end{proof}

\begin{theorem}\label{thrm:measure stability}
    Vectorization of $\hat{M}_1(X,A)$ through a locally Lipschitz finite Schauder basis is stable and injective. 
\end{theorem}

\begin{proof}
    Let $\B= \{f_i\}_{i=0}^\infty$ be a l.l.f. Schauder Basis of $Lip_c(X,A)$, with upper Lipschitz bound $M$. \\
    
    Stability: First we consider the case that measures be unsigned. Let $\alpha, \beta \in \hat{\mathcal{M}}^+_1(X,A)$. Let $\pi$ be a partial matching of $\alpha$ and $\beta$ and let $A = \{i \in \mathbb{N} | \ \alpha(f_i) \geq \beta(f_i)\}$  and $B = \{i \in \mathbb{N}\ | \ \alpha(f_i) < \beta(f_i)\}$. Then, \\

    \begin{equation*}
        \begin{split}
            ||F(\alpha) - F(\beta)||_1&= \sum\limits_{i=0}^\infty |\alpha(f_i) - \beta(f_i)|\\
            &= \sum\limits_{i \in A} \left(\alpha(f_i) - \beta(f_i)\right) + \sum\limits_{i \in B} \left(\beta(f_i) - \alpha(f_i)\right)\\
            &=  \alpha \left(\sum\limits_{i \in A} f_i\right) - \beta \left(\sum\limits_{i \in A} f_i\right)  + \left( \beta\left(\sum\limits_{i \in B} f_i\right) - \alpha \left(\sum\limits_{i \in A} f_i\right)\right)\\
            &= \pi \left(\left(\sum\limits_{i \in A} f_i \oplus \left(- \sum\limits_{i \in A} f_i\right) \right) \right) + \pi \left(\left(\sum\limits_{i \in B} f_i\right) \oplus \left(-\sum\limits_{i \in B} f_i\right)\right) \textbf{  \ref{lem: couplings}}\\
            &\leq  \pi \left(\left| \sum\limits_{i \in A} f_i \oplus \left(- \sum\limits_{i \in A} f_i\right)  \right| \right) + \pi \left|\left(\sum\limits_{i \in B} f_i\right) \oplus \left(-\sum\limits_{i \in B} f_i\right|\right) \\
            &\leq  4M \cdot \pi(d(-,-))\\
            &= 4M \cdot Cost(\pi)
        \end{split}
    \end{equation*}

    Taking infemum over all partial matchings gives us, $||F(\alpha) - F(\beta) ||_1 \leq 4M \cdot W_1(\alpha, \beta)$. 

    To extend this to signed measures $\alpha, \beta \in \hat{\mathcal{M}}_1(X,A)$, recall that\\
    $W_1(\alpha, \beta ):= W_1(\alpha^++\beta^-, \beta^++\alpha^-)$. Then applying what we have shown above to unsigned measures $\alpha^++\beta^-$ and $ \beta^++\alpha^-$, 

    \begin{equation*}
        \begin{split}
          4M \cdot W_1(\alpha, \beta)& \geq ||F(\alpha^++\beta^-) - F(\beta^++\alpha^-)||_1\\
          &=||F(\alpha^+) +F(\beta^-) - F(\beta^+) -F(\alpha^-)||_1\\
          &= ||F(\alpha) - F(\beta)||_1
        \end{split}
    \end{equation*}

    We now prove injectivity. 
    Let $\alpha, \beta \in \hat{\mathcal{M}}_1(X,A)$ such that $\alpha \neq \beta$. Since $\alpha, \beta$ are Radon measures, there exists compact $C\subset X$ such that $\alpha (C) \neq \beta(C)$. \\

    For $\epsilon >0$, define $1_C^\epsilon$ to be the $\frac{1}{\epsilon}$-Lipschitz approximation of the indicator function on $C$. That is, $1_C^\epsilon(x) = \max\{0, 1-\frac{d(x,C)}{\epsilon}\}$. \\

    Note that as we let $\epsilon \rightarrow 0$, $1_C^\epsilon$ decreases pointwise to $1_C$. Thus by Dominated Convergence Theorem, $\alpha^+ (1_C^\epsilon) \rightarrow \alpha^+(1_C)$ and $\alpha^- (1_C^\epsilon) \rightarrow \alpha^-(1_C)$. Hence $\alpha (1_C^\epsilon) \rightarrow \alpha(1_C)$. similarly, $\beta(1_C^\epsilon)\rightarrow \beta(1_C)$. Thus there must exist some $\epsilon >0$ such that $\alpha(1_C^\epsilon) \neq \beta(1_C^\epsilon)$. Fix such an $\epsilon$. \\

    Let $\{a_i^\epsilon\}_{i=0}^\infty  $ be such that $\sum\limits_{i=0}^\infty a_i^\epsilon f_i = 1_C^\epsilon$. \\

    Therefore, $\sum\limits_{i=0}^\infty a_i^\epsilon \alpha(f_i) = \alpha(1_C^\epsilon)$, and similarly $\sum\limits_{i=0}^\infty a_i^\epsilon \beta(f_i) = \beta(1_C^\epsilon)$ . Thus since $\alpha (1_C^\epsilon) \neq \beta(1_C^\epsilon)$, there must exist some $f_i$ such that $\alpha(f_i) \neq \beta(f_i)$. \\

\end{proof}

The stability result has a tighter bound when $\B$ is derived from a Coxeter-Freudenthal-Kuhn triangulation, analogous to Theorem \ref{thrm:optimal bound}. 

\begin{corollary}\label{cor: CFK measures}
     Let $\{(T^n,S^n)\}_{n=0}^\infty$ be a Coxeter-Freudenthal Kuhn nested triangulation on  $(X,A)$ and let $\mathbb{B}$ be the Schauder basis of $Lip_c(X,A)$ of Lipschitz n-linear functionals as in definition \ref{thrm: Schauder Basis}, with Lipschitz constants $(L_n)_{n=0}^\infty$, $\sum_{n=0}^\infty L_n =L$. Then for any two  $\alpha, \beta \in \hat{\mathcal{M}}_1(X,A)$, 
    $$||F_\B(\alpha) - F_\B(\beta)||_1 \leq \sqrt{2d}\cdot L \cdot  W_1(\alpha, \beta)$$

    Furthermore, this bound is optimal.
\end{corollary}

For a Schauder basis $\B$ derived from a nested triangulation of polyhedral pair $(X,A)$, we've shown that $F_\B$ has discriminating power on the set of persistence diagrams, and the superset of relative Radon measures. In fact, $\B$ is minimal for $F_\B$ to have discriminating ability over all of $\hat{M}_1(X,A)$.   That is, all functionals of the Schauder basis $\B= \{\K_v\}_{n \geq 0 \ : v \in V^{(n)}}$ are necessary for inducing an injective embedding.

\begin{theorem}\label{thrm:minimality}
    If $\B$ is a Schauder basis of $Lip_c(X,A)$ defined as in \ref{def:kernel} on a nested triangulation $\{(T^n, S^n)\}_{n=0}^\infty$. Then vectorization $F_\B: \hat{\mathcal{M}}_1(X,A) \rightarrow \ell_1$ of a Schauder basis formed as in definition \ref{def:kernel} is minimal. That is, for any $v \in V$, there exists $\alpha, \beta \in \hat{\mathcal{M}}_1(X,A)$, with $\alpha \neq \beta$ such that $\alpha(\K_{v'}) = \beta(\K_{v'})$ for all $v' \neq v$.

\end{theorem}

\begin{proof}
    Fix $v \in V^{(n)}$ for some $n \geq 0$. Let $\alpha  = \delta_v$ be the Dirac measure on $v$. We will build a measure $\beta \neq \alpha$ such that $\alpha(\K_{v'} ) = \beta(\K_{v'})$ for all $v' \neq v$. 

    First note that for any $m >n$, and $w \in V^{(m)}$, $\alpha(\K_w)= \K_w(v) =0$. So we consider layers $\leq n$.
    
    Let $\beta^n = 0$ be the zero measure.  For layer $n-1$, there exists a unique simplex in $T^{n-1}$  of minimal degree, $\sigma_{n-1} = \langle w_0, ...w_{p_{n-1}}\rangle$ such that $x \in \sigma_{n-1}$. Then  the only functionals on vertices in layer $n-1$ that are nonzero on $v$ are those centered at vertices of $\sigma_{n-1}$. Let $\beta^{n-1} = \sum\limits_{i =0}^{n-1} \frac{\K_{w_i}(v)}{\K_{w_i}(w_i)} \delta_{w_i}$. Then for each $w_i \in \sigma_{n-1}$ , $\alpha(\K_{w_i}) = \K_{w_i}(v) = \beta^{n-1}(\K_{w_i})$. Furthermore, for $w \in V^{(n-1)}$, with $w \notin \sigma_{n-1}$, $\alpha(\K_w) =0 = \beta(\K_w)$.

    Continuing inductively on layers of the triangulation below $n-1$, let $0 < i \leq n$. Then there exists a unique minimal simplex $\sigma_{n-i} = \langle w^{n-i}_0, ...w^{n-i}_{p_i}\rangle  \in T^{n-i}$ such that $v \in \sigma_{n-i}$. If $j >i$, then $n-j < n-i< n$, and hence $\sigma_{n-i}$ is contained in $\sigma_{n-j}$. 

    Let $\beta^{n-i} = \beta^{n-i+1} + \sum\limits_{j=0}^{p_{n-i}} \frac{\K_{w_j}(v)- \beta^{n-i+1}(\K_{w_j})}{\K_{w_j}(w_j)} \delta_{w_j}$\\

    Then for each $w_j$, 

    \begin{equation*}
        \begin{split}
            \beta^{n-i} (\K_{w^{n-i}_j}) &=  \beta^{n-i+1}(\K_{w^{n-i}_j}) + \frac{\K_{w^{n-i}_j}(v) -\beta^{n-i+1}(\K_{w^{n-i}_j}
            )}{\K_{w^{n-i}_j}(w^{n-i}_j)} \delta_{w_j}(\K_{w^{n-i}_j})\\
            &= \beta^{n-i+1}(\K_{w_j}) + \K_{w^{n-i}_j}(v) -\beta^{n-i+1}(\K_{w_j})\\
            &= \K_{w^{n-i}_j}(v)\\
            &= \alpha(\K_{w^{n-i}_j})
        \end{split}
    \end{equation*}

    Thus we may continue this process iteratively, increasing $i$ through all layers $\leq n$. Then  $\beta := \beta^0$ will meet the desired criteria. That is, $\alpha \neq \beta$, and  $\alpha (\K_{v'}) = \beta(\K_{v'})$ for all $v' \in V$, with $v' \neq v$. Hence, the Schauder basis on vertices of a triangulation is minimal for distinguishing between relative Radon measures on $(X,A)$.

\end{proof}

 \begin{corollary}
           We may again recall the Schauder basis $\B$ of stacked functionals of Definition \ref{def:stacked functional}. For unsigned measure $\alpha \in \hat{M}^+_1(X,A)$, $||F_\B(\alpha)||_1 = W_1(\alpha, \emptyset)$.
   \end{corollary}

The proof of this is analogous to the proof of Proposition \ref{thrm:stacked schauder eval}.

\printbibliography
\end{document}